\documentclass[10pt]{amsart}
\usepackage{amssymb,amsmath,amscd,amsfonts,amsthm,mathrsfs, enumerate}
\usepackage{verbatim} 
\usepackage{pdfsync}
\usepackage{tikz}
\input xy
\xyoption{all}
\usepackage[all]{xy}
\usepackage{color}

\usepackage[latin1]{inputenc}

\newcommand{\N}{\mathbb{N}}

\newcommand{\R}{\mathbb{R}}
\newcommand{\C}{\mathbb{C}}

\newcommand{\PR}{\mathbb{P}}

\newcommand{\E}{\mathbb E}
\def \lint{[\![}
\def \rint{]\!]}
\newcommand{\parent}{\overline}

\newcommand{\Cc}{\mathcal{C}}
\newcommand{\Dc}{\mathcal{D}}

\newcommand{\Tc}{\mathcal{T}}
\newcommand{\Id}{{\rm{Id}} }

\newcommand{\ran}{\mathrm{Ran}}  
\newcommand{\supp}{\mathrm{supp}}  

\newcommand{\Ar}{\mathscr{A}}

\newcommand{\Dr}{\mathscr{D}}
\newcommand{\Er}{\mathscr{E}}
\newcommand{\Fr}{\mathscr{F}}
\newcommand{\Gr}{\mathscr{G}}
\newcommand{\Hr}{\mathscr{H}}
\newcommand{\Kr}{\mathscr{K}}

\newcommand{\Nr}{\mathscr{N}}

\newcommand{\Tr}{\mathscr{T}}
\newcommand{\Qr}{\mathscr{Q}}
\newcommand{\Ur}{\mathscr{U}}
\newcommand{\Vr}{\mathscr{V}}
\newcommand{\Wr}{\mathscr{W}}

\newcommand{\ve}{\varepsilon}

\renewcommand{\deg}{{\rm deg}}

\def\build#1_#2^#3{\mathrel{\mathop{\kern 0pt#1}\limits_{#2}^{#3}}}

\newtheorem{theorem}{Theorem}[section]
\newtheorem{proposition}[theorem]{Proposition}
\newtheorem{lemma}[theorem]{Lemma}
\newtheorem{remark}[theorem]{Remark}
\newtheorem{example}[theorem]{Example}
\newtheorem{definition}[theorem]{Definition}
\newtheorem{corollary}[theorem]{Corollary}

\numberwithin{equation}{section}

\newcommand{\beq}{\begin{equation}}
\newcommand{\eeq}{\end{equation}}

\newcommand{\beas}{\begin{eqnarray*}}
\newcommand{\eeas}{\end{eqnarray*}}

\def \bone{\mathbf{1}}
\def \bd{{\rm{bd}}}

\title[Essential spectrum and Weyl asymptotics]{Essential spectrum and
  Weyl asymptotics for discrete Laplacians} 

\begin{document}
\author{Michel Bonnefont}
\address{Institut de Math\'ematiques de Bordeaux, Universit\'e
Bordeaux, $351$, cours de la Lib\'eration
\\$33405$ Talence cedex, France}
\email{michel.bonnefont@math.u-bordeaux.fr}
\author{Sylvain Gol\'enia}
\address{Institut de Math\'ematiques de Bordeaux, Universit\'e
Bordeaux, $351$, cours de la Lib\'eration
\\$33405$ Talence cedex, France}
\email{sylvain.golenia@math.u-bordeaux.fr}
\subjclass[2010]{34L20,05C63, 47B25, 47A63, 47A10}
\keywords{discrete Laplacian, locally finite graphs, 
asympotic of eigenvalues, spectrum, essential spectrum, markov chains,
functional inequalities}
\date{Version of \today}
\begin{abstract}
In this paper, we investigate spectral properties of discrete Laplacians.
Our study is based on the Hardy inequality and the use of
super-harmonic functions.  
We recover and improve  lower bounds for the bottom of the
spectrum and of the essential spectrum. In some situation, we  obtain Weyl
asymptotics for the eigenvalues. 
We also  provide a probabilistic representation of super-harmonic functions. 
Using coupling arguments, we set comparison results for  the bottom of
the spectrum,   the bottom of the essential spectrum and the
stochastic completeness  of different discrete Laplacians.  The class
of weakly spherically symmetric graphs is also studied in full
detail. 
\end{abstract}
\maketitle

\tableofcontents

\section{Introduction}

The study of discrete Laplacians on infinite graphs is at the
crossroad of spectral theory  and geometry. 
A special role is played by the bottom of the spectrum and that of the
essential spectrum of discrete Laplacians. Concerning the former, a
famous link is given through 
Cheeger/isoperimetrical inequalities, e.g.,
\cite{BHJ, BKW, D1, D2, DK,  Fu,M1,M2, KL2,Woj}. 
For the latter, since the essential spectrum can be thought as the
spectrum of the Laplacian ``at infinity'', the link is given through
isoperimetrical inequalities at infinity, e.g., \cite{Fu, KL2}.  In this article
we tackle the question with another standpoint and 
establish a new link with the help of Hardy inequalities, see Section
\ref{s:3}, and  positive super-harmonic functions. 

We fix briefly some notation. 
A weighted graph $\Gr$ is a triple  
 $\Gr:=(\Vr,\Er,m)$, where $\Vr$ denotes a countable set (the
 \emph{vertices} of $\Gr$),  $\Er$ a non-negative symmetric function
 on $\Vr  \times \Vr$  and $m$ a positive function on $\Vr$. We say
 that two 
 points $x,y \in \Vr$ are \emph{neighbors}  and we denote $x\sim y$ if
 $\Er(x,y)=\Er(y,x)>0$. We assume that $\Gr$ is locally finite in the
 sense  that each point of $\Vr$ has only  a finite  number of
 neighbors.  

The Laplacian then reads, for $f$ with finite support,
\[
\Delta_m f(x) = \frac{1}{m(x)} \sum_{y,y\sim x} \Er(x,y) (f(x)-f(y)). 
\]
We then consider its Friedrich extension and keep the same symbol. It
defines a non-negative self-adjoint operator. Its spectrum is thus
included in $[0,\infty)$.

Given $W:\Vr\to (0,+\infty)$, the Hardy inequality reads as follows:
\[
 \langle f, \Delta_m f \rangle_m \geq \langle
f, \frac{\tilde \Delta_m W}{W} f \rangle_m, 
\]
where $f:\Vr \to \R$ with finite support and $\tilde \Delta_m$ denotes
the the algebraic Laplacian (see Section 2.1). The heart of this
article is to exploit this inequality for some good choice(s) of
$W$. This method is very flexible:

\begin{itemize}
\item We  recover and improve lower bounds
for the spectrum and for the essential spectrum, see Section \ref{s:4}.
\item We improve some
criteria for the absence of the essential spectrum, see Section \ref{s:4}.
\item We  study
the eigenvalues below the 
essential spectrum   and obtain  Weyl asymptotics for the eigenvalues,
see Section \ref{s:5}.
\item We state an Allegretto-Piepenbrink type theorem for the
spectrum and the essential   spectrum, see Section \ref{s:6}. This
theorem links the bottom of the spectra with the existence of positive
super-harmonic functions.
\item We establish a probabilistic representation of
super-harmonic functions, see Section \ref{s:8}. As a corollary, we derive
a probabilistic understanding of 
the bottom of the spectrum and of the essential spectrum. 
\item For \emph{weakly spherically
    symmetric graphs} we prove that  the bottom of
  the spectrum and of the essential spectrum are governed only by the
  radial part of the Laplacian, see Section \ref{s:9}. 
\item  Using a coupling argument, we establish new
 comparison results for the bottom of the spectrum 
 and  the essential spectrum of different discrete Laplacians, see
 Section \ref{s:10}.
\item We derive a comparison result for the
 stochastic completeness, see Section \ref{s:11}. 
\end{itemize}
A main part of our work is to provide geometric criterion to ensure
the existence of positive super-harmonic functions; that is functions
$W: \Vr \to (0,\infty)$  satisfying 
\[
\tilde \Delta_m W(x) \geq \lambda(x) W(x),
\]
where $\lambda: \Vr \to [0,\infty)$ is some non-negative
function. These criterion are based on the geometric properties of a
1-dimensional decomposition of the graph. 
In many situations, this 1-dimensional decomposition is given by the
distance to a point or a finite set. 
Using min-max principles,  we then derive the lower bounds for the
bottom of  both the  spectrum and the essential spectrum  
(see Theorems \ref{t:lyap1}, \ref{thm-spectrum-bd} and Corollary
\ref{cor-ess-1}).  We also derive  lower bounds for the eigenvalue.

The Hardy inequality was already known in this discrete setting under
different names, e.g., \cite{CTT,  Go2, HK}.
Our present work is inspired by the work of \cite{CGWW}, where the
authors study diffusion operators  in a continuous setting and with a
finite invariant measure.  They  give criteria based on Lyapunov
functions to show that the Super-Poincar\'e Inequality holds.   

We mention that the Super-Poincar\'e Inequality was introduced
by Wang (see \cite{Wan1,Wan2,Wan3}) and is  equivalent to  a lower
bound of the essential spectrum. We do not rely on this approach but
this point of view enlightens about the situation. For the sake 
of completeness, we include the proofs of their results in Appendix
\ref{section-SPI}.  

The Hardy inequality directly gives one direction of the
Allegretto-Piepenbrik theorem (Theorem \ref{thm-reverse}). For the
other direction, knowing a lower bound on the spectrum or the
essential spectrum, one has to construct a positive super-harmonic
function. This was known for the spectrum   (e.g., \cite{HK})
but seems to be  new for the essential spectrum.

We then provide a  probabilistic representation of   
super-harmonic functions (see Theorems \ref{hitting2} and
\ref{hitting-cont}). It is interesting to compare with \cite{CGZ}.
The difference is that they control how the
stochastic process returns in a compact domain whereas we control how
the associated Markov process goes to infinity. An important tool is
the Harnack inequality that we borrow from \cite{HK}, see Section \ref{s:7}.

Next, we prove comparison results for  the bottom of the spectrum and
the essential spectrum of different weighted  Laplacians. Theorem
\ref{thm-comparaison-general}  
is an important improvement of Theorem 4 in \cite{KLW}. 
The main new ingredient in the proof of  Theorem
\ref{thm-comparaison-general} is a  coupling argument between the
associated stochastic processes  
(see Proposition \ref{good-coupling}). This coupling argument works
under a condition we called \emph{stronger weak-curvature growth}
which is strictly weaker than the \emph{stronger curvature growth}
condition of  Theorem 4 in \cite{KLW}.  Moreover, we  treat the
case of the essential spectrum. The coupling argument also
provides a comparison result for stochastic completeness (see Theorem
\ref{thm-comparaison-SC}).

Besides, we  study the class of \emph{weakly spherically symmetric graphs}, 
see Definition \ref{def-wss}. These graphs are a
slight generalization of the ones introduced in \cite{KLW}. We first
show weakly spherically symmetric graphs are exactly the graphs  such that
the radial part of their associated continuous time Markov chain is
also a 1-dimensional continuous time  Markov chain,  (see 
Propositions  \ref{prop-m-1d} and \ref{prop-bd-1d}).  
We then show  that  both the bottom of the spectrum and the essential
spectrum for the Laplacian on a weakly spherically symmetric graph
coincide with  the ones of their radial part  
(see Theorem \ref{cor-inf-sym}).

The paper is organized as follows. In Section \ref{s:2}, we present
the notation and we carefully introduce the  Laplacian.    
Section \ref{s:3} is devoted to the statement and a new proof of
the Hardy inequality. The lower bounds for the spectrum and the
essential spectrum are  obtained in  Section \ref{s:4}.  
In Section \ref{s:5}, we focus on eigenvalues. The estimates for the
eigenvalues  are very dependent of the intrinsic geometry of the
graphs.  The case of radial trees is given in Theorem
\ref{t:radial-tree}. Theorem 
\ref{prop-eigenvalue-gen-wss} treats the case of general weakly
spherically symmetric graphs.  
In  Section \ref{s:6}, we state and prove the Allegreto-Piepenbrik
type theorem (Theorem \ref{thm-reverse}).  
Section \ref{s:7} is devoted  to Harnack inequalities  for
positive super-harmonic functions. 
The construction of the discrete and continuous time Markov
chain associated to the Laplacian are made  in Section \ref{s:8}. We
also provide  
the  probabilistic representation of    super-harmonic functions (see
Theorems \ref{hitting2} and \ref{hitting-cont}). Section \ref{s:9}
is dedicated to the  study of the class of \emph{weakly spherically
  symmetric graphs}.   
In Section \ref{s:10}, using a coupling argument, 
we compare  the bottom of
the spectrum and the essential spectrum of two given weighted
Laplacians.  Section \ref{s:11} deals with stochastic completeness.
The construction of the Friedrichs extension of the Laplacian is
recalled in Appendix \ref{s:friedrichs} and Appendix \ref{section-SPI}
is devoted to the Super-Poincar\'e inequality.

\section{Notation}\label{s:2}

\subsection{The Laplacian on a graph}
Let us  consider  
 a graph $\Gr:=(\Vr,\Er,m)$ where $\Vr$ denotes a countable set of
 vertices of $\Gr$,  $\Er$ a non-negative symmetric function on $\Vr
 \times \Vr$  and $m$ a positive function on $\Vr$. We say that two
 points $x,y \in \Vr$ are neighbors   if
 $\Er(x,y)=\Er(y,x)>0$. In this case we write $x\sim y$. 
We assume that $\Gr$ is locally finite in the
 sense  that each point of $\Vr$ has only  a finite  number of
 neighbors. For simplicity, we also assume that each connected
 component of  $\Gr$ is infinite.

Let $\ell^2(\Vr,m)$ be the set of  functions $f:\Vr \to \C$ such that 
\[
\Vert f \Vert_{\ell^2(\Vr,m)} :=\sum_{x\in \Vr} |f(x)|^2 m(x).
\]
 $\ell^2(\Vr,m)$ is an Hilbert space with respect to the scalar product:
\[
 \langle f,g \rangle_m := \sum_{x\in \Vr} \overline{f(x)}g(x) m(x) \textrm{ for } f,g \in \ell^2(\Vr,m).
\]
For all $f,g\in \Cc_c(\Vr)$,  we introduce the quadratic form
\[
 Q(f,g):=\frac{1}{2} \sum_x \sum_y \Er(x,y) \overline{\left(f(x) - f(y)\right)}\left(g(x)-g(y)\right).
\]
Note that $Q$ is well-defined since the graph is locally finite.
This quadratic form  is non-negative, i.e., 
$Q(f,f)\geq 0$ for all $f\in \Cc_c(\Vr)$ and closable. There is a
unique self-adjoint operator $\Delta_\Gr$ such that
\[Q(f,f)= \langle f, \Delta_\Gr f\rangle\]
for all $f\in \Cc_c (\Vr)$ and $\Dc(\Delta_\Gr^{1/2})$ is the
completion of $\Cc_c (\Vr)$ under the norm $\|\cdot\| +
Q(\cdot,\cdot)^{1/2}$. We refer to Appendix \ref{s:friedrichs} for its
construction. We have:
 \begin{equation}\label{Lm}
\Delta_{\Gr} f (x) =  \frac{1}{m(x)} \sum_{y\in \Vr} \Er(x,y)
\left(f(x)-f(y)\right), \mbox{ for all }f\in \Cc_c (\Vr).
\end{equation} 
We call this operator the \emph{Laplacian} associated to the graph $\Gr$.
According to the context, we will also denote it by   $\Delta_{\Gr,m}$
or $\Delta_m$. We mention that  $\Delta_{\Gr}$ is the Friedrichs
extension of $\Delta_{\Gr}|_{\Cc_c(\Vr)}$. 

Note that $\Delta_\Gr|_{\Cc_c(\Vr)}$ is not necessarily essentially
self-adjoint. We refer to \cite{Go2} for a review of this matter.

We write with the symbol  $\tilde \Delta_\Gr$ the algebraic version of
$\Delta_\Gr$, i.e.,
\begin{equation*}
\tilde \Delta_{\Gr} f (x) =  \frac{1}{m(x)} \sum_{y\in \Vr} \Er(x,y)
\left(f(x)-f(y)\right), \mbox{ for all }f:\Vr \to \C.
\end{equation*} 
Recall that $\tilde \Delta$ is well-defined since $\Gr$ is locally finite.

\subsection{1-dimensional decomposition, distance function and
  degrees}

A \emph{1-dimensional decomposition} of the graph $\Gr:=(\Vr,\Er)$ is
a family of finite sets $(S_n)_{n\geq 0}$ which forms a partition of
$\Vr$, that is $\Vr=\sqcup_{n\geq 0} S_n$, and such that for all $x\in S_n,
y\in S_m$,  
\begin{eqnarray*}
 \Er (x,y)>0 \implies |n-m|\leq 1.
\end{eqnarray*} 
Given such a 1-dimensional decomposition, we write $|x|:=n$ if $x\in
S_n$. We also write $B_N:=\cup_{0\leq i\leq N} S_i$. A function
$f:\Vr\to \R$ is said to be \emph{radial}  if $f(x)$ depends only on
$|x|$.

Typical examples of such a 1-dimensional decomposition are given by
level sets of the graph  distance function to a finite set $S_0$ that
is     
\begin{equation}\label{e:1ddSn}
S_n:=\{x \in \Vr, d_\Gr(x,S_0)=n\},
\end{equation}
where the \emph{graph distance function} $d_\Gr$ is defined by  
\begin{equation}\label{distance}
 d_\Gr(x,y):= \min\{n\in \N, x\sim x_1\sim \dots
\sim x_n=y, x_i \in \Vr, i=1,\dots ,n\}. 
\end{equation} 
Note that for a general 1-dimensional decomposition, one has solely $d_\Gr(x,
S_0)\geq n$,  for $x\in S_n$. Given $x\in S_n$ and $k\geq-n$, we
  shall denote by
\begin{equation}\label{e:skx}
S_{k,x}:= \{y\in S_{n+k},\, d_\Gr(x,y)=|k|\}.
\end{equation}
We introduce the following \emph{unweighted degrees} of $x\in S_n$:
\[ \eta_0(x):=\sum_{y\in S_n
  }\Er(x,y),
\quad 
\eta_\pm(x):=\sum_{y\in S_{n\pm 1}}\Er(x,y), \]
with the convention that $S_{-1}=\emptyset$, i.e., $\eta_-(x)=0$ when
$x\in S_0$. The \emph{total unweighted degree} of $x$ is defined by:
\[\eta(x):= \eta_0(x)+\eta_-(x)+\eta_+(x)=  \sum_{y\in\Vr} \Er(x,y).\]
We stress that contrary to $\eta_\pm$ and $\eta_0$, $\eta$ is
independent of the choice of $1$-dimensional decomposition.  

We now divide by the weight and obtain new quantities of interest. 
We  call them the  \emph{(weighted) degrees} and  denote them by:
\begin{align*}
\deg_a(x):=&\frac{\eta_a(x)}{m(x)}, \quad \mbox{ where } a\in\{0, -,
+\},
\\
\deg(x):=&\frac{\eta(x)}{m(x)}= \deg_-(x)+\deg_0(x)+\deg_+(x).
\end{align*}
Again, note that $\deg$ is independent of the choice of a
$1$-dimensional decomposition. When $m(x)=\eta(x)$, we also write
$p_+(x),p_0(x),p_-(x)$ for $\deg_+(x),\deg_0(x),\deg_-(x)$,
respectively.   

In the same spirit, we also define:
\[
\deg(x,y):=\frac{\Er(x,y)}{m(x)}.
\]

We say that the graph $\Gr$ is \emph{simple} when  $ \Er:\Vr \times \Vr
\to \{ 0,1\}$. This definition is independent of the choice of the
weight $m$. 
In this case, when  $m=1$, the operator $\Delta_1$  is usually called
the \emph{combinatorial Laplacian} on the graph $\Gr$ whereas when
$m(x)= \eta (x)$,  
or equivalently when $\deg \equiv 1$, the operator $\Delta_\eta$ is
usually called the \emph{normalized Laplacian}.  

In the case of the combinatorial Laplacian $\Delta_1$ on a simple
graph, one has 
\[\deg_\pm(x)=\eta_\pm(x)= \# \{y, y\sim x, |y|=|x|\pm 1\}\]
and \[\deg_0(x)=\eta_0(x)=  \#  \{y, y\sim x, |y|=|x|\}.\]

Given a function $V:\Vr\to \C$, we
denote  by $V(\cdot)$ the operator of multiplication by $V$. It is
elementary that
$\Dc(\deg^{1/2}(\cdot))\subset\Dc(\Delta^{1/2}_{m})$. Indeed, one has:  
\begin{align}\nonumber
\langle f, \Delta_{m} f\rangle_m&= \frac{1}{2} \sum_{x\in \Vr} \sum_{y\sim
  x} \Er(x,y)|f(x)-f(y)|^2
\\
\label{e:majo}
&\leq \sum_{x\in \Vr} \sum_{y\sim 
  x} \Er(x,y)(|f(x)|^2+|f(y)|^2) = 2\langle f, \deg(\cdot) f\rangle_m,
\end{align}
for $f\in\Cc_c(\Vr)$. 
This inequality also gives a necessary condition for the absence of
essential spectrum for $\Delta_{m}$ (see \cite{Go2}[Corollary
2.3]). In \cite{Go2}[Proposition 4.5], we also prove 
that, in general, the constant $2$ cannot be improved. It
is also easy to see 
that $\Delta_{m}$ is bounded if and only if $\deg$ is
(e.g. \cite{Go2, KL1, KLW}).

\subsection{Weakly spherically symmetric graphs}
We introduce the class of weakly spherically symmetric graphs. Their
associated Laplacian will be study deeply in 
Section \ref{s:9} and \ref{s:10}.

\begin{definition}\label{def-wss}
Let  $\Gr:=(\Vr,\Er,m)$ be a weighted graph and let $(S_n)_{n\geq 0}$ be a  1-dimensional decomposition on $\Gr$.
We say that $\Gr$ is \emph{weakly spherically symmetric} with
    respect to $(S_n)_{n\geq 0}$ if the quantities $\deg_+(x)$ and
    $\deg_-(x)$ only depend on the quantity $|x|$. 
\end{definition}

It is easy to see that for a weakly spherically symmetric graph, the 1-dimensional decomposition  corresponds to the one obtained by taking the level sets of the distance function to the set $S_0$; that is we have
\[
S_n=\{x\in\Vr,d_\Gr(x,S_0)=n\}.
\]
This due to the fact that for $x\in S_n$, obviously, one has $\deg_-(x)>0$ thus $d_\Gr(y,S_{n-1})=1$.

If $S_0=\{x_0\}$,  we also say that  $\Gr=(\Vr,\Er,m)$ is
\emph{weakly spherically symmetric around} $x_0$. The definition
\ref{def-wss} is a slight generalization of the one in \cite{KLW}
where the authors only consider the case of weakly spherically
symmetric around a point $x_0$.  
In \cite{KLW}, the author shows  that  weakly spherically symmetric
graphs with $S_0=\{x_0\}$ are exactly the graphs such that the heat
kernel associated to $\Delta_m$, $p_t(x_0,\cdot)$, is a radial
function.  

The interest of our definition is that if $\Gr:=(\Vr,\Er,m)$ is weakly
spherically symmetric with respect to some 1-dimensional decomposition
$(S_n)_{n\geq 0}$ then so is the induced graph with vertex set $\Vr -
B_n$, $n\geq 0$. In particular, in Proposition \ref{prop-m-1d}, we
prove that with our definition,  weakly spherically symmetric graphs
correspond exactly to the graphs such that the radial part of the
associated continuous time Markov chain is still a continuous time
Markov chain.

\subsection{Decomposition of the Laplacian and bipartite
  graphs}\label{sec-decomp-laplacien}
We fix a weighted graph $\Gr:=(\Vr, \Er, m)$ and $(S_n)_{n\in \N}$ a
$1$-dimensional  decomposition of $\Vr$ 
We decompose the Laplacian in the following way:
\begin{equation}\label{decomp-laplacien}
\Delta_m  = \deg(\cdot) - A_{m,{\rm bp}}-A_{m,{\rm sp}},
\end{equation}
where
\[
 A_{m,{\rm bp}} f(x):= \frac{1}{m(x)}\sum_{y,|y|\neq |x|} \Er(x,y)f(y)\]
and 
\[A_{m,{\rm sp}} f(x) :=\frac{1}{m(x)}\sum_{y,|y|= |x|} \Er(x,y)f(y).
\]
Here $bp$ and $sp$ stand for bi-partite and spherical, respectively. 

 Let $U$ be unitary operator defined by $U f(x):=
(-1) ^{|x|} f(x)$, then  
\begin{equation}\label{decomp-laplacien2}
U \Delta_m U  = \deg(\cdot) + A_{m,{\rm bp}} -A_{m,{\rm sp}} = 2
\deg(\cdot) - \Delta_m -2 A_{m,{\rm sp}}.
\end{equation}
Note if $\eta_0 \equiv 0$ thus
$\Delta_m= \deg(\cdot) - A_{m,{\rm bp}}$ and 
\[U \Delta_m U=
\deg(\cdot)+ A_{m,{\rm bp}}= 2\, \deg(\cdot)  - \Delta_m.
\]
In particular, when $m=\eta$, 
\[U \Delta_\eta  U= 1+ A_{\eta,{\rm bp}} = 2- \Delta_\eta.\]
In this last case, this directly yields:
\begin{proposition}\label{prop-spectrum-bp}
Let $\Gr:=(\Vr,\Er,m)$ be a bi-partite graph with $m=\eta$. Let $\Wr$ be
any subset  of  $\Vr$, then 
the spectrum $\sigma \left(\Delta_\eta^\Wr\right)$ is symmetric with
respect to $1$.  
\end{proposition}

\subsection{The Dirichlet Laplacian and Persson's lemma}

Let $\Ur$ be any subset of $\Vr$.
First, we define $\textrm{Int}\, \Ur:=\{x\in \Ur, y\sim x \Rightarrow y\in
\Ur\}$ the interior of $\Ur$ and  $\partial \Ur:=\{x\in \Ur, \exists y \in \Ur^c
, y\sim x\}$ the boundary of $\Ur$. 

 We call $\Gr^\Ur:=(\Ur,\Er^\Ur,m)$ the \emph{induced graph} on $\Ur$ where
 $\Er^\Ur$ is defined on $\Ur\times \Ur$ by $\Er^\Ur(x,y):=\Er(x,y), x,y\in
 \Ur$. 

We denote by $\Delta_\Gr^\Ur$ the associated \emph{Dirichlet
  Laplacian}. It is defined as follows:   for $f: \Ur \to \C$ with
compact support, we define $\tilde f:  \Vr\to \C$ by $\tilde f
(x)=f(x)$, if $x\in \Ur$ and $\tilde f =0$ otherwise, we set: 

\[
 \Delta_\Gr^\Ur  f (x):= \Delta_\Gr \tilde f(x), \mbox{ for all } x\in \Ur.
\]
Note that $\Delta_\Gr^\Ur$ is a self-adjoint operator acting in
$\ell^2(\Ur)$. It is the Friedrichs extension of
$\Delta_\Gr^\Ur|_{\Cc_c(\Ur)}$.

The infimum of the essential spectrum of $\Delta_m$ is classically
described by the 
Persson Lemma, e.g., \cite{KL1}[Proposition 18].  One reads: 
\begin{align}\nonumber
\inf \sigma_{\rm ess} (\Delta_\Gr)&= \sup_{\Kr\subset \Vr
  \rm{finite}} \inf \sigma\left( \Delta_\Gr^{\Kr^c}\right)
\\ \label{e:Persson}
&= \sup_{\Kr\subset \Vr
  \rm{finite}}\quad \inf_{f\in \Cc_c(\Vr\setminus   \Kr), \|f\|=1} \langle f,
\Delta_\Gr f\rangle.  
\end{align}
Note that if $\Delta_\Gr$ is bounded from above we also have
\begin{align*}\nonumber
\sup \sigma_{\rm ess} (\Delta_\Gr)&= \inf_{\Kr\subset \Vr
  \rm{finite}} \sup \sigma\left( \Delta_\Gr^{\Kr^c}\right)
\\ \label{e:Persson}
&= \inf_{\Kr\subset \Vr
  \rm{finite}}\quad \sup_{f\in \Cc_c(\Vr\setminus   \Kr), \|f\|=1} \langle f,
\Delta_\Gr f\rangle.  
\end{align*}

We adapt the Upside-Down-Lemma of \cite{BGK} which was inspired from
\cite{DM}.  

\begin{lemma}[Upside-Down-Lemma]\label{l:upsidedown1}
 Let $\Gr:=(\Vr, \Er, m)$  be a weighted graph, $q:\Vr\to \R$ and $\Ur \subset
 \Vr$. Assume  there are $ a\in(0,1),$ $ k\geq  0$ such that for all
 $f\in \Cc_c(\Ur)$,  
\begin{align*}
(1- a)\langle f, (\deg+q)(\cdot) f \rangle_m  - k\Vert
f\Vert^2_m \leq  \langle f,\Delta^\Ur_\Gr f + q(\cdot) f \rangle_m,
\end{align*}
then  for all $f\in \Cc_c(\Ur)$, we also have
\begin{align*} \nonumber
 \langle f,\Delta^\Ur_\Gr f+ q(\cdot) f  \rangle_m  \leq  (1+ a)\langle f, 
 (\deg+q)(\cdot) f \rangle_m +  k\Vert f\Vert^2_m.
\end{align*}
\end{lemma}
\proof By a direct calculation we find for $f\in\Cc_c(\Ur)$
\begin{align*}
    \langle f, (2\,\deg(\cdot)-\Delta_\Gr^\Ur)f\rangle_m&=\frac{1}{2}\sum_{x,y\in\Vr,x\sim
      y}\Er(x,y)(2|f(x)|^{2}+ 2|f(y)|^{2})-|f(x)-f(y)|^{2})\\
    &=\frac{1}{2}\sum_{x,y,x\sim y} \Er(x,y)|f(x)+f(y)|^{2}
\\
&  \geq\frac{1}{2}\sum_{x,y,x\sim y}\Er(x,y)\left| |f(x)|-|f(y)|
    \right|^{2}
\\
&=\langle |f|, \Delta_\Gr^\Ur|f|\rangle_m.
\end{align*}
Using the assumption gives after reordering
\begin{align*}
   \langle f, \Delta_\Gr^\Ur f +q(\cdot)  f\rangle_m-\langle f,
   (2\deg+q)(\cdot)  f\rangle_m&\leq 
   -\langle|f|, \Delta_\Gr^\Ur |f|\rangle_m\\
   &\leq-(1- a)\langle |f|, (\deg+q)(\cdot) |f|\rangle_m + k\langle
   |f|,|f|\rangle_m\\
   &=-(1- a)\langle f,  (\deg+q)(\cdot)f\rangle_m + k\langle  f,f\rangle_m
\end{align*}
which yields the assertion.\qed

Combining the upside-down Lemma and the Persson criteria we
derive immediately the following proposition.
\begin{proposition}\label{p:upsidedowness}
Let $\Gr:=(\Er, \Vr)$ be a graph. Assume that there is $a>0$ such that
\[\inf \sigma_{\rm ess}(\Delta_\eta) \geq 1-a\]
then 
 \[\sup \sigma_{\rm ess}(\Delta_\eta) \leq 1+a.\]
\end{proposition}

\section{Hardy inequality and its links with super-harmonic
  functions}\label{s:3} 

In this paper, one major tool is the following Hardy inequality. The
terminology comes from \cite{Go2}. The idea
is to bound the Laplacian from below by
a potential and to reduce its analysis to it. This technique has
already be successfully used in \cite{HK} for some ground state
related problem and in \cite{Go2}  to obtain some Weyl asymptotic.

\begin{proposition}[Hardy inequality]\label{prop:hardy}
Let $W$ be a positive function on $\Vr$, then for all $f\in \Cc_c(\Vr)$, 
\begin{equation}\label{Hardy}
\mathcal Q(f,f) = \langle f, \Delta_m f \rangle_m \geq \langle
f, \frac{\tilde \Delta_m W}{W} f \rangle_m. 
\end{equation}
\end{proposition}

Here  we recall that $\tilde \Delta_m W$ has to be understood in a algebraical sense since in general $W$ is
general not a $\ell^2$ function. We mention that there are other
techniques to bound the Laplacian from below by a potential and refer
to \cite{CTT,   CTT2, MilTru}.

The inequality of Proposition \ref{prop:hardy}  is well-known in the continuous setting. 
 It can be seen as an integrated version of
Picone's identity (see for example \cite{allegretto}). It also appears in the work \cite{CGWW}.

We point out that the formulation of \eqref{Hardy} is equivalent to
the one used in \cite{HK, Go2}. We shall present an alternative proof,
which is closer to the one of \cite{CGWW}.  We shall
only use the reversibility of the measure $m$.

\proof  Take $f\in \Cc_c(\Vr)$,
\begin{align*}
\left\langle f , \frac{\tilde \Delta_m W}{W} f\right\rangle_m 
                                       &=  \sum_x \sum_y \Er(x,y)
                                       \left(|f|^2(x)
                                         -\frac{W(y)}{W(x)}|f|^2(x)\right)\\
&=\sum_x \sum_y \Er(x,y) \left(|f|^2(x) -
  \frac{1}{2}\left(\frac{W(y)}{W(x)} |f|^2(x) +
    \frac{W(x)}{W(y)}|f|^2(y)\right)\right). 
\\
&\leq \sum_x \sum_y
\Er(x,y) \left(|f|^2(x) - \Re \left(\overline f(x)f(y)\right)\right)\\
&= \frac{1}{2} \sum_x \sum_y \Er(x,y) \left|f(x) - f(y)\right|^2=
\mathcal Q(f,f). 
\end{align*}
This is the announced result. \qed

The aim of this work is to investigate the links between some properties of the spectrum
of the Laplacian and the existence of some positive
function $W$  which satisfies 
\begin{equation}\label{Dmlsh}
\frac{\tilde \Delta_m W}{W} (x) \geq \lambda(x), 
\end{equation} 
for all $x\in \Vr$ and for some function $\lambda$ which is
non-negative away from a compact.  Clearly, given $m,m':\Vr\to
(0,+\infty)$, a function $W$ 
 satisfies (\ref{Dmlsh}) for  $m$ if and only if it satisfies
 (\ref{Dmlsh}) for  $m'$ where:
\[\frac{\tilde \Delta_{m'} W}{W} (\cdot) \geq \psi(\cdot), \mbox{ with
}\psi(\cdot)=\frac{m(\cdot)}{m'(\cdot)}\lambda(\cdot). \]
This simple fact enlightens about the flexibility of our method.

Note that in the literature, when $\lambda$ is constant, these
functions $W$ are sometimes called  positive $\lambda$-super-harmonic
functions. In a different field, they are also called Lyapunov
functions. We rely on the next definition.

\begin{definition}
A positive function $W$ is called a \emph{Lyapunov function} if  there exist
$\lambda$ a positive function,  $b>0$, and a finite set $B_{r_0}$ such
that for all $x\in \Vr$,  
\beq \label{lyap1}
\frac{\tilde \Delta_m W}{W} (x)\geq  \lambda(x) - b \bone_{B_{r_0}}(x). 
\eeq

A positive function $W$ is called a \emph{super-harmonic function} if
there exist 
$\lambda$ a non-negative  function such
that for all $x\in \Vr$,  
\beq \label{superharmonicfunction}
\frac{\tilde \Delta_m W}{W} (x)\geq  \lambda(x) . 
\eeq
\end{definition}

\begin{remark} Usually, for Lyapunov functions, the condition $W\geq
  1$ is  also required and they are used to control the return time in
  a compact region (see \cite{CGWW}). Here we do the contrary and our Lyapunov
  functions  control how the process goes to infinity (see section
  \ref{s:8}). They are  
  non-increasing  in our applications. Therefore, we shall not impose that
  $W \geq 1$.  
\end{remark}

Moreover, in some situations, we will have to consider family of
super-harmonic functions. We set:

\begin{definition}\label{d:exhau}
Given a graph $\Gr:=(\Er,\Vr)$, we call a sequence $(\Vr_n)_n$ of
finite and connected subsets of  $\Vr$ \emph{exhaustive} if 
$\Vr_n\subset \Vr_{n+1}$,  and $\cup_n \Vr_n=\Vr$. 
\end{definition}

\begin{definition}
Set a graph $\Gr:=(\Er,\Vr)$. 
A family of positive functions $(W_n)_n$ is called
\emph{a family of super-harmonic functions relative to an exhaustive
  sequence} $(\Vr_n)_n$   if there exists a non-zero and non-negative function
$\lambda:\Vr\rightarrow \R^+$ such that
\beq \label{lyapf1}
\Delta_{\Gr} W_n (x)\geq  \lambda(x) W_n(x),
\eeq
for all $x\in \Vr_n$.
\end{definition}

\section{Super-harmonic functions,  essential  spectrum, and  
minoration of  eigenvalues}\label{s:4}   

In this section, we construct Lyapunov  and super-harmonic functions
for the Laplacian on some weighted graphs  and study the (essential)
spectrum of the associated 
Laplacian. We compare our approach with the ones obtained by
isoperimetrical techniques and provide some minoration of the
eigenvalues which are below the essential spectrum. 

\subsection{A few words about the isoperimetrical approach}

Given a function $m:\Vr\to(0, \infty)$ and $U\subseteq \Vr$, we define the
\emph{isoperimetric constant} as follows:
\[ 
\alpha_m (U):=\inf_{K,K \subseteq U\subset \Vr} \frac{L(\partial
  K)}{m(K)}, \mbox{ where } L(\partial
  K):=\langle  \bone_{\partial K},
  \Delta_m \bone_{\partial K} \rangle_m=\langle  \bone_K,
  \Delta_m \bone_K \rangle_m.\]
Note that $L(\partial  K)$ is independent of $m$.  Trivially, one has
that $\alpha_m (U)\geq \inf \sigma(\Delta_{m}^U)$. However, it is
important to notice that this quantity is also useful to estimate from
below the Laplacian.  One obtains in \cite{KL2}[Proposition 15] (see
also \cite{D1, DK, Kel} and references therein), the following result.  

\begin{proposition}[Keller-Lenz]\label{p:Keller-Lenz}
Given $\Gr:=(\Vr, \Er, m)$, then
\begin{align}\label{e:Keller-Lenz}
\inf \sigma\left(\Delta_{m}^U\right) \geq d_U\left(1 -
  \sqrt{1-\alpha_\eta(U)^2}\right), 
\end{align}
where $d_U:= \inf_{x\in U}\deg(x)$. Moreover, if $D_U:=\sup_{x\in
  U}\deg(x) <+\infty$, we obtain:
\[
\inf \sigma\left(\Delta_{m}^U\right) \geq \left({D_U} -
  \sqrt{D_U^2-\alpha_m(U)^2}\right). 
\]
\end{proposition} 
It remains to estimate the isoperimetric constant. We adapt
straightforwardly the proof of 
\cite{Woj}[Theorem 4.2.2], where the author considered the case
$w=\eta$.   
\begin{proposition}\label{p:Woj422}
Take a graph $\Gr:=(\Vr, \Er, w)$, $w:\Vr\to (0, \infty)$ and
$U\subset \Vr$. Suppose that there are a 1-dimensional decomposition of $\Gr$ and  $a>0$ such that
\[\eta_+(x)-\eta_-(x) \geq a w(x),\]
for all $x\in U$, then one obtains that $\alpha_{w}(U)\geq a$. 
\end{proposition} 
\proof Set $r(x):= |x|$. We have $\Delta_w r(x) \leq -a$ for $x\in U$. By the Green Formula and since
$r(x)-r(y)\in \{0, \pm 1\}$ for $x\sim 
y$,  we have: 
\begin{align*}
L(\partial K)&= \sum_{x\in K, y\sim x, y\notin K} \Er(x,y)\geq
\left|\sum_{x\in K, y\sim x, y\notin K} \Er(x,y)(r(x)-r(y))\right|
\\ &=
\left|\langle \bone_K ,\Delta_w r\rangle_w \right|\geq a w(K).
\end{align*}
This yields the result. \qed

\subsection{Lower estimates of eigenvalues} In the continuous setting, it is possible from the Hardy inequality and the Super-Poincar\'e
Inequality (see the Appendix) to obtain some estimates of the heat semigroup and 
 then to obtain some eigenvalues comparison.
 Here in this discrete setting, the situation is simpler since 
bounding from below the Laplacian by a non-negative multiplication
operator directly give information on eigenvalues.  
In all this section we denote by 
\[0\leq \lambda_1(\Delta_m) \leq \lambda_2(\Delta_m) \leq \dots \leq \lambda_n(\Delta_m) \leq
\dots <\inf \sigma_{\rm ess}(\Delta_m)
\]
the eigenvalues of $\Delta_m$ which are located below the infimum of
the essential spectrum of $\Delta_m$. A priori this number of
eigenvalues can  be finite.
 We recall some well-known results. We refer to \cite{RS}[Chapter XIII.1] and \cite{DS}
for more details and to \cite{RS}[Chapter XIII.15] for more
applications. We start with the form-version of the standard
variational characterization of the $n$-th eigenvalue.

\begin{theorem}\label{t:min-max}
 Let $A$ be a non-negative self-adjoint operator with form-domain
$\Qr(A)$. For all $n\geq 1$, we define:
\begin{align*}
\mu_n(A):=\sup_{\varphi_1, \ldots, \varphi_{n-1}}\inf_{\psi\in [\varphi_1,
  \ldots, \varphi_{n-1}]^\perp} \langle \psi,
A \psi \rangle, 
\end{align*}
where $[\varphi_1,  \ldots, \varphi_{n-1}]^\perp = \{\psi\in \Qr(A),$ so
that $ \|\psi\|=1$ and   $\langle \psi, \varphi_i\rangle=0, $ with $
i=1, \ldots, n\}$. Note that 
  $\varphi_i$ are not required to be linearly independent. 

We define also:
\begin{align*}
\nu_n(A)&:= \inf_{E_n \subset \Qr(A) , \dim E_n=n} \sup_{\psi \in E_n, \Vert \psi \Vert=1} \langle \psi ,A \psi \rangle. 
\end{align*}

Then, 
one has $\mu_n(A)=\nu_n(A)$ and
if $\mu_n(A)=\nu_n(A)$ is (strictly) below the essential spectrum of $A$, it is the
$n$-th eigenvalue, counted with multiplicity, $\lambda_n(A)$. Moreover, we have that:
\begin{align*}
\dim \ran\, \bone_{[0, \mu_n(A)]}(A)=n. 
\end{align*}

Otherwise, $\mu_n(A)=\nu_n(A)$ is the
  infimum of the essential spectrum. Moreover,
  $\mu_{j}(A)=\nu_j(A)=\mu_n(A)=\nu_n(A)$, for all $j\geq n$ and there  are at most $n-1$
  eigenvalues, counted with multiplicity, below the essential
  spectrum. In that case, 
\begin{align*}
\dim \ran\, \bone_{[0,  \mu_n(A)+
  \varepsilon]}(A)= +\infty, \mbox{ for all }\varepsilon>0.
\end{align*}
\end{theorem}
 
This ensures the following useful criteria.

\begin{proposition}\label{p:compa}
Let $A,B$ be two self-adjoint operators, with form-domains
$\Qr(A)$ and $\Qr(B)$, respectively. Suppose that
\[\Qr(A)\supset \Qr(B) \mbox{ and } 0\leq \langle \psi, A\psi\rangle
\leq 
\langle \psi, B \psi\rangle, \]
for all $\psi \in \Qr(B)$. Then one has $\inf \sigma_{\rm ess} (A)\leq
\inf \sigma_{\rm ess} (B)$ and 
\begin{align}\label{e:N}
\Nr_\lambda(A) \geq \Nr_\lambda(B), 
\mbox{ for } \lambda \in [0,\infty)\setminus \{\inf \sigma_{\rm ess} (B)\},
\end{align}
where $\Nr_\lambda(A):=\dim \ran\, \bone_{[0, \lambda]}(A)$.  

In particular, if $A$ and $B$ have the same
form-domain, then $\sigma_{\rm ess} (A)=\emptyset$ if and only if
$\sigma_{\rm ess} (B)=\emptyset$ and  $\lambda_n(A) \leq \lambda_n(B)$, $n\geq 1$.
\end{proposition}

\proof It is enough to notice that $\mu_n(A)\leq \mu_n(B)$, for all
$n\geq 0$. Theorem \ref{t:min-max} permits us to conclude for the
first part. Supposing now they have the same form-domain, by the uniform
boundedness principle, there are $a,b>0$ such that:
\[\langle \psi, A\psi\rangle \leq a \langle \psi, B \psi\rangle +
b\|\psi\|^2 \mbox{ and } \langle \psi, B\psi\rangle \leq a \langle
\psi, A \psi\rangle + b\|\psi\|^2\]
for all $\psi \in \Qr(A)=\Qr(B)$. By using  the previous
statement twice we get the result.  \qed

We start with a direct application. We shall present examples in the
next section.

\begin{corollary}\label{cor-vp-minoration-min-max}
Let $\psi$ be a non-decreasing non-negative radial function on $\Vr$. 
Assume that
\begin{equation*}
\langle f,\Delta_m f \rangle_m \geq \langle f, \psi f \rangle_m,
\end{equation*}
for all $f\in \Cc_c(\Vr)$. Then, 
\[\inf \sigma_{\rm ess}(\Delta_m) \geq \lim_{|x|\to \infty} \psi(x)\] 
and when $\lambda_{|B_{n-1}|+k}(\Delta_m)$ exists, we have:
\begin{equation}\label{vp-minoration}
\lambda_{|B_{n-1}|+k}(\Delta_m) \geq \psi(n), \quad  \textrm{ for } k=1,\dots
,|S_n|. 
\end{equation}
\end{corollary}

\subsection{Upper estimates of eigenvalues}

It is also possible to obtain some upper bounds for the eigenvalues. 
Our method here is based on the following well-known Proposition, see
\cite{Wan3}[Proposition 5.1]  for example. 

\begin{proposition}
Let $\Gr:=(\Vr,\Er,m)$ be a graph and let $\Delta_m$ be the associated Laplacian.
 Let $g_1, \ldots, g_n \in \Dc(\Delta_m^{1/2})$ be $n$ orthonormal
 functions ($\langle g_i,g_j\rangle_m= \delta_{ij})$). 
Let $\lambda_n (M_g)$ be the largest eigenvalue of the symmetric
 matrix:  
\[
M_g:=\left({\langle g_i,\Delta_m g_j \rangle_m} \right)_ {1\leq i,j \leq n}. 
\]
Then if $\lambda_n(\Delta_m)$ exists we have:
\begin{equation}\label{Rayleigh-g-matrice}
\lambda_n(\Delta_m) \leq \lambda_n(M_g).
\end{equation}
 In particular, for all non identically zero functions $g_i$,
 $i=1,\dots ,n$ such that  
\begin{equation}\label{cond-Rayleigh}
\langle g_i, g_j \rangle_m =\langle g_i,\Delta_m g_j \rangle_m =0
\textrm{ for }i\neq j, 
\end{equation}
\begin{equation}\label{Rayleigh-g}
\lambda_n(\Delta_m) \leq \max_{i=1,\dots,n} \frac{\langle g_i,\Delta_m g_i
  \rangle_m}{\langle g_i, g_i \rangle_m}. 
\end{equation}
Moreover if $\lambda_n(M_g)<\inf \sigma_{\rm ess}(\Delta_m)$ then
$\lambda_n(\Delta_m)$ exists. 
\end{proposition}

\begin{proof}
This is a direct consequence of  Theorem \ref{t:min-max}. One  just
has  to note that for $E_n$ a subspace of dimension $n$ of  $
\Dc(\Delta_m)$ and $(g_1,\dots, g_n)$ an orthonormal basis of $E_n$,
one has: $\max_{h\in E_n, \Vert h \Vert=1} \langle h,\Delta_m h
\rangle=\lambda_n(M_g)$. 
\end{proof}

As a corollary, we obtain:

\begin{corollary}\label{cor-dist-support}
 \begin{enumerate}
 \item Let $g_1, \ldots , g_n \in \Dc(\Delta_m^{1/2})$ be
   such that  $d_\Gr(\supp\, 
   g_i, \supp\,  g_j)\geq 2$, for $i\neq 
   j$ and where $\supp$ denotes the support.  Then if
   $\lambda_n(\Delta_m)$ exists we have:
\[
\lambda_n(\Delta_m) \leq \max_{1=1,\dots,n} \frac{\langle g_i,\Delta_m g_i
  \rangle_m}{\langle g_i, g_i \rangle_m}. 
\]
\item Let $\Gr_i:=(\Vr_i,\Er_i,m)$,  $i=1,\dots ,n$  be $n$ induced
  sub-graphs of $\Gr$ such that for  $i\neq j$, $d_\Gr(\Vr_i,\Vr_j)\geq
  2$, then  if
   $\lambda_n(\Delta_{\Gr,m})$ exists we have:
\[
\lambda_n(\Delta_{\Gr,m}) \leq  \max_{i=1,\dots,n} \left\{ \inf \sigma
  \left(\Delta_{\Gr,m}^{\Gr_i}\right) \right\} 
\]
where $\Delta_{\Gr,m}^{\Gr_i}$ denotes the Dirichlet Laplacian of
$\Gr_i$ in $\Gr$. 
\end{enumerate}
\end{corollary}

\subsection{The approach with super-harmonic functions}
In this section we improve a result of
\cite{Woj} and prove that
the weighted Laplacian $\Delta_m$ has empty essential spectrum
for a certain class of graph and give some estimation on the
eigenvalues.  

\begin{theorem}\label{t:lyap1} Take $\Gr:=(\Er, \Vr, m)$ and assume there is a $1$-dimensional
decomposition and a constant $c>1$ such that
\beq\label{cond-lyap-sl}
l:=\liminf_{|x|\to  \infty} \left(\deg_+(x) -  c\,  \deg_-(x)\right)>0
\eeq
Set $n_0:= \inf\{n\in \N, \, \deg_+(x) -  c\,  \deg_-(x)\geq 0$ with
$|x|\geq n \}$. Then there exists a super-harmonic function $W$ such that   
\begin{align}\label{e:phic0}
\tilde \Delta_m W (x)\geq \phi_c W(x), \quad \mbox{ for all } x\in \Vr,
\end{align}
with 
\begin{align}\label{e:phic}
\phi_c(x):= \frac{c-1}{c} (\deg_+(x)-c \, \deg_-(x))
\bone_{B_{n_0}^c}\geq 0.
\end{align}
In particular, we obtain that $\sigma_{\rm ess} ( \Delta_m)\geq
l(c-1)/c$ and $\sigma_{\rm ess} ( \Delta_m)=\emptyset$ if $l=\infty$.
\end{theorem}

\begin{proof}
We construct a suitable Lyapunov function. Set $\tilde W(x):=c^{-|x|}$.
We have:
\begin{align}\label{e:psic}
\psi_c(x):= \frac{\tilde \Delta_m \tilde W (x)}{\tilde W(x)}  = \deg_+(x)
 \left(1-\frac{1}{c}\right)  + \deg_-(x) \left(1-c\right). 
\end{align}
Now since $\deg_+(x)- c\, \deg_-(x)$  is positive outside a given ball $B_{n_0}$,
$\tilde W$ is a  Lyapunov function, which satisfies
\[
\tilde\Delta_m \tilde W (x)\geq \phi_c (x) \tilde W(x) - C \bone_{ B_{n_0}} 
\] 
with $\phi_c$ defined as in \eqref{e:phic} and for some constant $C$.  

Set now 
\begin{align}\label{e:goodW}
 W(x)= \left\{ \begin{array}{cl}
    {c^{-n_0}},& \textrm{ if } x\in B_{n_0} \\ 
    {c^{-|x|}},& \textrm{ if } x\in B_{n_0}^c,  \end{array}
 \right.
\end{align}  
then it satisfies \eqref{e:phic0}. Finally, since
$\phi_c(x) $ tends to $l(c-1)/c$ when $|x|\to 
\infty$, Corollary \ref{cor-lyap-spec} gives the statement about the
essential spectrum. \end{proof}

\begin{remark}
Note that condition \eqref{cond-lyap-sl} with $l=+\infty$ is equivalent
to the following one: 
$\eta_+ (x) \to \infty \textrm{  as } |x|\to \infty$  and there exist
a ball $B_{n_0}$ and a constant $c'>0$ such that for all  $x$\ outside
the ball $B_{n_0}$,  
\[ 
 \frac{\deg_+(x) - \deg_-(x)}{ \deg_-(x)} \geq c'>0.
\]
Thus, when $m=1$, this is better than  the one of \cite{Woj}[Theorem
4.2.2] which asserts:  $\eta_{+} (x) \to \infty \textrm{  as }
|x|\to \infty$  and  there exist  a ball $B_{n_0}$ and a constant
$c>0$ such that for all  $x$\ outside the ball $B_{n_0}$,   
\[
 \frac{\eta_+(x) - \eta_-(x)}{ \eta(x)} \geq c>0.
\]
His result follows for instance by Propositions
\ref{p:Keller-Lenz}, \ref{p:Woj422} and the Persson Lemma
\eqref{e:Persson}. 
\end{remark}
An example where our criterion is satisfied and the one of
Wojciechowski is not satisfied is the following. 

\begin{example}
 Let $\Gr:=(\Vr,\Er,1)$ be the simple graph  with weight $m\equiv 1$, whose set of vertices  is 
\[
\Vr:=\{ (1,i_1,i_2,\dots ,i_k), k\geq 0, i_j\in \lint 1,  j\rint\mbox{ for } j\in \lint
1, k\rint\}
\]
and   where $\Er(x,y):=1$  if and only if $x\neq y$ and 
\[
\{x,y\} =\{(1,i_1,\dots ,i_k) ,(1,i_1,\dots ,i_k,i_{k+1}) \}\]
or
$x=(1,i_1,\dots ,i_k) \textrm{ and } y=(1,i'_1,\dots ,i'_k)$.

For  $x=(1,i_1,\dots ,i_k)\in \Vr, k\geq1$, we have
$\eta_+(x)= k+1$, $\eta_-(x)=1$ and $\eta(x)= k+ 2 + k!-1$.  
\end{example} 

\renewcommand\a{1}
\renewcommand\b{1}
\renewcommand\c{2}
\begin{figure}
\begin{tikzpicture}
\draw[color=orange , ultra thick](3,1)--(-3,1);
\draw[color=orange , ultra thick](-3.5,2)--(3.5,2);
\draw[color=orange , ultra thick](-3.5-1/6,3)--(3.5+1/6,3);
\fill[color=black](0,0)circle(.7mm);
\fill[color=black](0,-1)circle(.7mm);
\draw (0,-1)--(0,0);
\foreach \x in {{-\a},\a} {
\draw (0,0)--(3*\x,1);
\fill[color=black](3*\x,1)circle(.7mm);
\foreach \y in {{-\b},...,\b} {
\fill[color=black](3*\x+\y/2,2)circle(.5mm);
\draw (3*\x,1)--(3*\x+\y/2,2);
\foreach \z in {-2, -1, 1, 2} {
\fill[color=black](3*\x+\y/2+\z/15,3)circle(.3mm);
\draw (3*\x+\y/2,2)--(3*\x+\y/2+\z/15,3);
}
}
}
\path(-5, -1) node {$S_0$};
\path(-5, 0) node {$S_1$};
\path(-5, 1) node {$S_2$};
\path(-5, 2) node {$S_3$};
\path(-5, 3) node {$S_4$};
\end{tikzpicture}
\caption{{Growing tree with complete graph on spheres}}\label{f:ex10}
\end{figure}
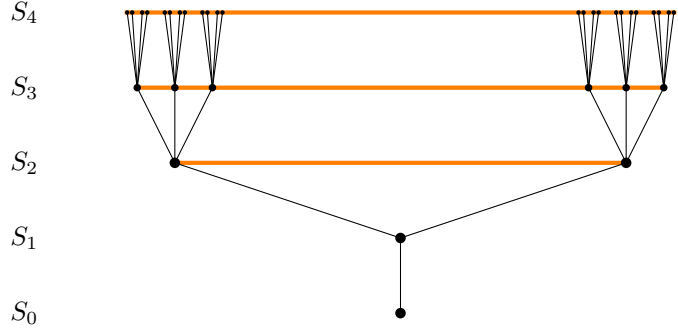

We provide an example of  a weakly spherically symmetric graph. On  a weakly spherically symmetric graph, $\inf \sigma_{\rm ess}(\Delta_1)$  does not depend on the edges inside the spheres $S_n$ (see Corollary \ref{cor-spectrum-wss}). Therefore, it is a good point that our criterion \ref{cond-lyap-sl} does not depend on $\deg_0$.

Theorem \ref{t:lyap1} can also be useful to compute the asymptotics of
eigenvalues. We improve partially  the main result of
\cite{Go2} where one considered some perturbation of weighted trees.

\begin{theorem}\label{t:lyap2}
Take $\Gr:=(\Er, \Vr, m)$ and assume there is a $1$-dimensional
decomposition such that 
\begin{align}\label{e:asymp-lyap0}
\lim_{|x|\to \infty}\deg_+(x)= \infty, \quad \mbox{ and } \quad 
\max(\deg_-(x), \deg_0(x))= o(\deg_+(x)),
\end{align}
as $|x|\to \infty$, then $\Dc(\Delta_m^{1/2})=\Dc(\deg^{1/2}(\cdot))$, 
$\sigma_{\rm ess} ( \Delta_m)=\emptyset$,  and 
\begin{equation}\label{e:asymp-lyap1}
\lim_{n \to \infty} \frac{\lambda_n(\Delta_m)}{\lambda_n(\deg(\cdot))}=1.
\end{equation}
\end{theorem}

\begin{proof}
We apply Theorem \ref{t:lyap1}. Note first that $l=\infty$ for all
$c>1$. The essential spectrum of $\Delta_m$ is therefore empty. 
Using \eqref{e:phic}, \eqref{Hardy} and \eqref{e:asymp-lyap0} we obtain that for all
$\varepsilon>0$ there are $c_\varepsilon, c_\varepsilon'>0$ such that:
\begin{equation}\label{e:sparse}
  \langle f, \Delta_m f\rangle \geq (1-\varepsilon)\langle f,
\deg(\cdot)f\rangle_m - c_\varepsilon \|f\|^2_m,
\end{equation}
for all $f\in \Cc_c(\Vr)$. Combined with \eqref{e:majo}, we get the
equality of the form domains. Using Lemma \ref{l:upsidedown1} we derive:
\begin{align*}
\langle f, \Delta_m f\rangle  \leq (1+\varepsilon)\langle f,
\deg(\cdot)f\rangle_m +  c_\varepsilon  \|f\|^2_m,
\end{align*}
for all $f\in \Cc_c(\Vr)$. This yields: 
\[ 1- \varepsilon \leq  \liminf_{n \to
  \infty}\frac{\lambda_n(\Delta_m)}{\lambda_n(\deg(\cdot))}
\leq \limsup_{n \to
  \infty}\frac{\lambda_n(\Delta_m)}{\lambda_n(\deg(\cdot))}\leq 1+
\varepsilon.\]  
By letting $\varepsilon$ go to zero we obtain the Weyl asymptotic
\eqref{e:asymp-lyap1} for $\Delta_{m}$. 
\end{proof}

\begin{remark}  Inequalities $\eqref{e:sparse}$ was studied in full
  detail in \cite{BGK}. It turns out that the graphs which satisfy
  $\eqref{e:sparse}$ are exactly the so-called \emph{almost sparse
    graphs} (see  the definition in \cite{BGK}). Combining Proposition
  \ref{p:Woj422} and \cite[Theorem 5.5]{BGK} we can also reprove
  Theorem \ref{t:lyap2}. 
\end{remark}

With the same method, we also obtain a result when the inner and
  outer degrees are bounded.   
 
\begin{theorem}\label{thm-spectrum-bd}
Take $\Gr:=(\Er, \Vr, m)$ and assume there is a $1$-dimensional
decomposition such that there exist $n_0\in \N$  and two constants $a$
and $D$ with 
\begin{equation}\label{cond-lyap-lin}
  \deg_+(x) - \deg_-(x) \geq a    \textrm{ for all } x \in B_{n_0}^c
\end{equation}
and 
\begin{equation}\label{cond-lyap-lin2}
\sup_{x\in B_{n_0}^c}  \deg_+(x) + \deg_-(x) \leq D < +\infty.
\end{equation}
Then there exists a positive function $W$ such that  
\begin{align*}
\tilde \Delta_m W (x)\geq \phi_c W(x), \quad \mbox{ for all } x\in \Vr,
\end{align*}
with 
\begin{align*}
\phi_c(x):= \left(D - \sqrt{D^2-a^2}\right) \bone_{B_{n_0}^c}\geq 0.
\end{align*}
In particular,  $\inf
\sigma\left(\Delta_m^{B_n^c}\right) 
\geq D - \sqrt{D^2-a^2}$, for all $n\geq n_0$ and 
\begin{equation}
 \inf \sigma_{\rm ess}(\Delta_m) \geq D - \sqrt{D^2-a^2}. 
\end{equation}
\end{theorem} 
 
 \begin{proof}
Let $\tilde W(x)=c^{-|x|}$ for some $c>1$ which will be precised later.
Take $\psi_c$ as in \eqref{e:psic}. For $x\in
B_{n_0}^c$,  conditions (\ref{cond-lyap-lin}) and
(\ref{cond-lyap-lin2}) imply that:
\begin{align*}
\psi_c(x) &=  \frac{1}{2} \left[ \left(c-\frac{1}{c}\right)
  (\deg_+(x)-\deg_-(x))- \left(c+\frac{1}{c}-2\right) (\deg_+(x)
  +\deg_-(x)) \right]\\ 
          &\geq \frac{1}{2} \left[ \left(c-\frac{1}{c}\right) a-
            \left(c+\frac{1}{c}-2\right) D \right]= \frac{1}{2} \left[
            2D  -c (D-a) - \frac{1}{c}(D+a)\right]
\\
&\geq  D - \sqrt{D^2-a^2},
\end{align*}
by taking $c= \sqrt \frac{D+a}{D-a}$. Then by choosing $W$ as in
\eqref{e:goodW}, Corollary \ref{cor-lyap-spec}  ends the proof. 
 \end{proof}

An  example where our criterion is satisfied and Proposition
\ref{p:Keller-Lenz} does not apply is the following:  

\begin{example}
Given $d\geq 2$, let $\Gr:=(\Vr,\Er,1)$ be the simple graph 
given by the $d$-ary tree with the complete graph
on each sphere and with weight $m\equiv1$, see Figure
\ref{f:ex14}. The graph is constructed as 
follows. The set of vertices  is  
\[\Vr:=\{ (1,i_1,i_2,\dots ,i_k), k \geq 1,
 i_j\in \lint 1, d\rint\mbox{ and } j\in \lint
1, k\rint\}\] 
and $\Er(x,y):=1$  if and only if $x\neq  
y$ and  
\[
\{x,y\} =\{(1,i_1,\dots ,i_k) ,(1,i_1,\dots ,i_k,i_{k+1}) \}
\]  
or 
\[
x=(1,i_1,\dots ,i_k) \textrm{ and } y=(1,i'_1,\dots ,i'_k).
\] 
We have $\# S_k=d^{k}$, for $k\geq 0$. Moreover, for $x=(i_0,\dots
,i_k)\in \Vr, k\geq 1$, we have $\deg_+(x)=\eta_+(x)= d$,
$\deg_-(x)=\eta_-(x)=1$, and $\deg (x)=\eta(x)= \#S_k+d$. By Theorem
\ref{thm-spectrum-bd}, we get:
\[
\inf \sigma(\Delta_1) \geq d+1-2\sqrt{d},
\]
whereas the lower bound given by \eqref{e:Keller-Lenz} is
$0$. Indeed, for $U$ the complement of a ball, by considering $K={S_n}$ for
$n$ large enough, one sees that $\alpha_\eta(U)=0$. Note also that the second part of 
Proposition \ref{p:Keller-Lenz} does not apply since $D_U=\infty$. 
\end{example}

\renewcommand\a{1}
\renewcommand\b{1}
\renewcommand\c{1}
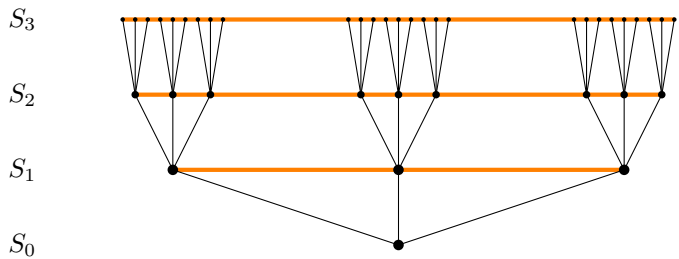
\begin{figure}
\begin{tikzpicture}
\draw[color=orange , ultra thick](3,1)--(-3,1);
\draw[color=orange , ultra thick](-3.5,2)--(3.5,2);
\draw[color=orange , ultra thick](-3.5-1/6,3)--(3.5+1/6,3);
\fill[color=black](0,0)circle(.7mm);
\foreach \x in {{-\a},...,\a} {
\draw (0,0)--(3*\x,1);
\fill[color=black](3*\x,1)circle(.7mm);
\foreach \y in {{-\b},...,\b} {
\fill[color=black](3*\x+\y/2,2)circle(.5mm);
\draw (3*\x,1)--(3*\x+\y/2,2);
\foreach \z in {{-\c},...,\c} {
\fill[color=black](3*\x+\y/2+\z/6,3)circle(.3mm);
\draw (3*\x+\y/2,2)--(3*\x+\y/2+\z/6,3);
}
}
}
\path(-5, 0) node {$S_0$};
\path(-5, 1) node {$S_1$};
\path(-5, 2) node {$S_2$};
\path(-5, 3) node {$S_3$};
\end{tikzpicture}
\caption{{$3$-ary tree with complete graphs on spheres}}\label{f:ex14}
\end{figure}

Actually, the result: $\inf \sigma(\Delta_1) = d+1-2\sqrt{d}$ was already known for the above example. Since it  is a weakly spherically symmetric graph,  
 by Corollary 6.7 in \cite{KLW},  the quantity $\inf \sigma(\Delta_1)$ does not depend on the edges inside the spheres $S_n$.  Therefore one can reduce to the case of the ordinary d-ary tree. 
\begin{remark}\label{rq-W2}
 In the case of the \emph{normalized} Laplacian, $m(x)=\eta(x)=\sum_y
 \Er(x,y)$, that is  $\deg\equiv 1$, we bring some new light to 
 \cite[Corollary 16]{KL2} (which improves the original  result of
 \cite{DK}) : 
\[ 
 \inf \sigma_{\rm ess} (\Delta_\eta) \geq 1-\sqrt{1-a^2} 
\]
For the $d$-ary tree, we  also recover the sharp estimate:
\[
   \inf \sigma(\Delta_\eta) \geq  
   1-\frac{2 \sqrt{d}}{d+1} .
\]
\end{remark}

\subsection{Rapidly branching graphs}
We now discuss the result of Fujiwara and Higushi (see \cite[p
196]{Fu}) concerning rapidly branching graphs. In
\cite[Corollary 4]{Fu} under the hypothesis that $\sigma_{\rm
   ess}(\Delta_\eta)=\{1\}$, the author proves the existence of an
 infinite  sequence of eigenvalues $\lambda_i\neq 1$ that converges to 
$1$. Fujiwara asks if there exist two sequences of eigenvalues that tends
respectively to $1^-$ and  to $1^+$. The answer is yes:

\begin{proposition}\label{p:Fuj}
 Let $\Gr:=(\Er, \Vr)$ be a graph such that $\sigma_{\rm
   ess}(\Delta_\eta)=\{1\}$. Then there exist two infinite sequences of
 eigenvalues $(\lambda_n^+)_{n\in \N}$ and $(\lambda_n^-)_{n\in \N}$ such that
 $\lambda_n^+ >1$ and $\lambda_n^-<1$ for all $n\in \N$. 
\end{proposition} 

 We refer to \cite{Fu} for the question of $1$ being an eigenvalue in the
case of a radial tree.

\proof Given
$x_0\sim y_0$ set $g_\alpha=\delta_{x_0}+\alpha \delta_{y_0}$. We have
\[ \frac{\langle g_\alpha, \Delta_\eta g_\alpha\rangle}{\langle
  g_\alpha, g_\alpha\rangle_\eta}=1-\mbox{sign}(\alpha) \frac{\Er(x_0,
  y_0)}{\sqrt{\eta(x_0)\eta(y_0)}},\]
where $\alpha$ is chosen to be $\pm \sqrt{\eta(x_0)/\eta(y_0)}$.  
The result follows from Corollary \ref{cor-dist-support} (applied to
$\pm \Delta_\eta$). 
 \qed

\begin{remark} Note that $\sigma_{\rm    ess}(\Delta_\eta)=\{1\}$
implies that that $(\lambda_n^+)_{n\in \N}$
  and $(\lambda_n^-)_{n\in \N}$ tend to $1$ by definition of the
  essential spectrum. Moreover our choice of test-functions and
  Corollary \ref{cor-dist-support} also imply that:  
\[\lim_{|x|\to \infty} \inf_{y\sim x} \frac{\eta(x)\eta(y)}{\Er^2(x,
  y)}=+\infty,\]
where $|x|$ is defined with respect to any choice of $1$-dimensional
decomposition. We point out that with the help of Corollary
\ref{cor-ess-1} it is easy to construct a simple graph such that
$\sigma_{\rm   ess}(\Delta_\eta)=\{1\}$ and  such that
$\liminf_{|x|\to \infty} 
\eta(x)<+\infty$, see Figure \ref{f:counterex}. 
\end{remark} 

\renewcommand\a{1}
\renewcommand\b{2}
\renewcommand\c{3}
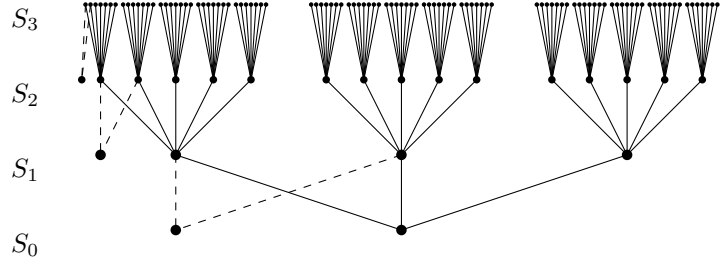
\begin{figure}
\begin{tikzpicture}
\fill[color=black](0,0)circle(.7mm);
\foreach \x in {{-\a},...,\a} {
\draw (0,0)--(3*\x,1);
\fill[color=black](3*\x,1)circle(.7mm);
\foreach \y in {{-\b},...,\b} {
\fill[color=black](3*\x+\y/2,2)circle(.5mm);
\draw (3*\x,1)--(3*\x+\y/2,2);
\foreach \z in {{-\c},...,\c} {
\fill[color=black](3*\x+\y/2+\z/15,3)circle(.3mm);
\draw (3*\x+\y/2,2)--(3*\x+\y/2+\z/15,3);
}
}
}
\draw[dashed](-3,1)--(-3,0);
\draw[dashed](0,1)--(-3,0);
\fill[color=black](-3,0)circle(.7mm);
\draw[dashed](-4,2)--(-4,1);
\draw[dashed](-3.5,2)--(-4,1);
\fill[color=black](-4,1)circle(.7mm);
\draw[dashed](-3-1-3/15,3)--(-4.25,2);
\draw[dashed](-3-1-2/15,3)--(-4.25,2);
\fill[color=black](-4.25,2)circle(.5mm);
\path(-5, 0-0.2) node {$S_0$};
\path(-5, 1-0.2) node {$S_1$};
\path(-5, 2-0.2) node {$S_2$};
\path(-5, 3-0.2) node {$S_3$};
\end{tikzpicture}
\caption{{Graph with $\sigma_{\rm   ess}(\Delta_\eta)=\{1\}$ and  
$\liminf_{|x|\to \infty} 
\eta(x)<+\infty$.}}\label{f:counterex}
\end{figure}

In the setting of simple graphs, the main result of  \cite{Fu} is the
equivalence between an isoperimetry at infinty and the fact that
$\sigma_{\rm    ess}(\Delta_\eta)=\{1\}$. We give a sufficient
condition for the latter.

\begin{corollary}\label{cor-ess-1}
Take $\Gr:=(\Er, \Vr)$ be a graph and assume there is a $1$-dimensional
decomposition such that : 
\begin{equation}\label{cond-ess-1}
p_+(x)\to 1, \quad \textrm{ as }|x| \to \infty,
\end{equation}
then $\sigma_{\rm ess} (\Delta_\eta)= \{1\}$.
\end{corollary}

\begin{proof}
 Let $\ve>0$. Since $p_+(x)\to1$  as $|x| \to \infty$, there exists a
 ball $n_\ve$ such that for all $n\geq n_\ve$, 
\[
p_+(x)-p_-(x)\geq 1-\ve \textrm{ for } x\in B_{n}^c.
\]
Thus, by Theorem \ref{thm-spectrum-bd}, we obtain $\inf \sigma
\left(\Delta_\eta^{B_n^c}\right)\geq 1 -\sqrt \ve$ and
Proposition \ref{p:upsidedowness} concludes.\end{proof}

\section{Eigenvalues Comparison}\label{s:5}
\subsection{The case of trees}
We turn to the case of a tree $\Tr:=(\Vr,\Er)$. First we fix $v\in \Vr$
and set $S_0:=\{v\}$ and $S_n$ given by \eqref{e:1ddSn}. 
Let $x\in \Vr$, we denote by $ \Tr_{x}'$ the induced tree in $\Tr$ whose
set of vertices is $\Vr_{\Tr_{x}'}= \cup_{k\geq 0} S_{k,x}$, see
\eqref{e:skx}. This corresponds to the sub-tree of $\Tr$ whose root is $x$.  
We also consider $\Tr_{x}$  the induced tree in  $\Tr$ whose set of
vertices is $\Vr_{\Tr_{x}'}= \cup_{k\geq -1} S_{k,x}$.  We denote by
$\parent x$ the unique point in $S_{-1,x}$, it is the father of $x$.

 Let $f \in \Cc_c(\Vr_{\Tr_{x}})$ such that $f(\parent x)=0$, we can
 extend $f$ in  a function $\tilde f $ on the all tree $\Tr$ by
 setting   $\tilde f (y)=0$ for all $y \in \Vr-\Vr_{n}$ and we have: 
\[
\Delta_{\Tr_{x},m}^{\Tr_{x}'} f (z) =\Delta_{\Tr,m} \tilde f (z)
\textrm{ for } z\in \Vr_{\Tr_{x}'}.
\]
If $\Tr$ is a radial tree, for all $x$ which belongs in a same sphere
$S_n$, all the sub-trees $\Tr_{x}$ (and $\Tr_{x} '$) are the same. We
denote by $\Tr_n$ one of them (and $\Tr_n'$ respectively). 

Corollary \ref{cor-dist-support} yields:
\begin{theorem}\label{thm-tree-vp-majoration}
Let $\Tr$ be a  tree. 
Then, with the above notations: 
\begin{equation}\label{vp-majoration}
\lambda_{\#S_n}(\Delta_m) \leq   \max_{x \in S_n} \inf \sigma
\left(\Delta_{\Tr_{x},m}^{\Tr_{x}'}\right). 
\end{equation} 
If moreover  $\Tr$ is a radial tree, then 
\begin{equation}\label{vp-majoration-radial}
\lambda_{\#S_n}(\Delta_m) \leq  \inf \sigma \left(\Delta_{\Tr_n,m}^{\Tr_n'}\right).
\end{equation} 
\end{theorem}

In the next propositions, we give a quantitative way to estimate the eigenvalues in the
case of radial  simple trees  for the normalized and the combinatorial Laplacians. 

In case of a simple tree, the term \emph{radial}  simply means  that the function $\eta$  only depends on the distance $|x|$. 
By abuse of notation, we denote $\eta(n):=\eta(x),x \in S_n$.

\begin{proposition}\label{prop-majoration-radial-tree}
Let $\Tr:=(\Vr,\Er)$ be a radial simple tree.
 Let $x\in S_n$, then with the above notation:
\[
\inf \sigma \left(\Delta_{\Tr_{x},\eta}^{\Tr_{x}'}\right) 
 \leq 1 - \sqrt \frac{  1-\frac{1}{\eta(n)} }   { \eta(n+1)}
\]
and 
\[
\inf \sigma \left(\Delta_{\Tr_{x},1}^{\Tr_{x}'}\right) 
 \leq  \max(\eta(n),\eta(n+1)) \left(1 - \sqrt \frac{  1-\frac{1}{\eta(n)} }   { \eta(n+1)}\right).
\]
\end{proposition}
\begin{proof}
We treat first the case of the normalized Laplacian.
Let $x \in S_n$.
 Let $g$ be the function on $\Vr_{\Tr_x}$ defined by $g(x):=1$,
 $g(y):=\alpha$ for $y\in S_{1,x}$, and $g(y):=0$ otherwise.
Clearly,
\begin{align*}
\inf \sigma \left(\Delta_{\Tr_{x},\eta}^{\Tr_{x}'}\right) &\leq  \frac{\langle g,\Delta_\eta  g \rangle_\eta}{\langle   g,  g \rangle_\eta}\\
                                              & = \frac{1 + (\eta(n)-1) (1-\alpha)^2 + (\eta(n)-1) (\eta(n+1)-1) \alpha^2}
                                                         {\eta(n)+ (\eta(n)-1) \eta(n+1) \alpha^2}\\
                                              &= 1 - \frac{2 \alpha (\eta(n)-1)}
                                                            {\eta(n)+ \alpha^2 (\eta(n)-1) \eta(n+1)}\\
                                              &= 1 - \frac{ \sqrt{ 1-\frac{1}{\eta(n)}}}
                                                            {\sqrt{
                                                                \eta(n+1)}},
\end{align*}
where in the last line we have made the choice $\alpha=\frac{\sqrt{\eta(n)}}{\sqrt {(\eta(n)-1) \eta(n+1)}}$.

The same computation gives also
\begin{align*}
 \inf \sigma \left(\Delta_{\Tr_{x},1}^{\Tr_{x}'}\right) &\leq  \frac{\langle g,\Delta_1  g \rangle_1}{\langle   g,  g \rangle_1} 
\leq \max(\eta(n),\eta(n+1)) \frac{\langle g,\Delta_\eta  g \rangle_\eta}{\langle   g,  g \rangle_\eta}.
\end{align*}
This ends the proof.
\end{proof}

We now precise the result of Fujiwara and Higushi (see \cite[p
196]{Fu} and Corollary \ref{cor-ess-1}) by estimating the eigenvalues
for the normalized Laplacian $\Delta_\eta$. We also discuss the case
of the combinatorial Laplacian. 

\begin{theorem}\label{t:radial-tree}\
 Let $\Tr:=(\Vr,\Er)$ be a simple radial tree.
\begin{enumerate}
\item Let $m=\eta$.  Assume $\eta(n)$ is non-decreasing and tends to $+\infty$  as $n$ tends to $\infty$. 
Then
\[
\sigma_{\rm ess} (\Delta_\eta)=\{1\}
\textrm{ and }
\sigma (\Delta_\eta)=\{1\} \cup  \{\lambda_i(\Delta_\eta),
2-\lambda_i(\Delta_\eta),  i\geq 1\}, 
\]
where $(\lambda_i(\Delta_\eta))_{i\geq1}$ is an infinite sequence of eigenvalues converging to $1$.
Moreover, for  $\ve>0$, 
and $n \geq n(\ve)$ we have:
\[
1-2  \sqrt \frac{1}{\eta(n)} \leq \lambda_{(\# B_{n-1}+1)}(\Delta_\eta) \leq  \lambda_{\#S_n}(\Delta_\eta)  \leq
 1 - \frac{ 1-\ve} {\sqrt{ \eta(n+1)}}.
\]

\item  Let $m=1$. Assume that $\eta(n) \left(1-2 \sqrt \frac{1}{\eta(n)}\right) $ is non-decreasing and that $\eta(n)$ tends to $+\infty$  as $n$ tends to $\infty$. 
Then,  
\[
\sigma_{\rm ess} (\Delta_1)=\emptyset
\textrm{ and } 
\sigma (\Delta_1)=  \{\lambda_i(\Delta_1), i\geq 1\}, 
\]
where $\lambda_i(\Delta_1)$ is an infinite sequence of eigenvalues which tends to $+\infty$. Moreover, one has:
\[
\eta(n) \left(1-2  \sqrt \frac{1}{\eta(n)}\right) \leq \lambda_{(\# B_{n-1}+1)}(\Delta_1) \leq \eta(n).\]
\end{enumerate}
\end{theorem}

\begin{proof}
We begin by the left inequality for the normalized Laplacian. Note that by hypothesis, for all $x\in B_{n-1}^c$, 
\[
\deg_+(x)-\deg_-(x)\geq \frac{\eta(n)-2}{\eta(n)}.
 \]
Therefore by Theorem \ref{thm-spectrum-bd} and by Corollary
\ref{cor-vp-minoration-min-max}, if the corresponding eigenvalue exists:
\begin{align*}
 \lambda_{(\#B_{n-1}+1)}(\Delta_\eta) &\geq 1-\left(\sqrt{1-
     \left(\frac{\eta(n)-2}{\eta(n)}\right)^2}\right) = 1- 2 \frac{
   \sqrt {1 -\frac{1}{\eta(n)}}} {\sqrt{ \eta(n)}} \geq 1-2  \sqrt
 \frac{1}{\eta(n)}. 
\end{align*}
Since $\eta$ tends to $\infty$, we have $\#B_{n-1}+1 \leq \#S_n$,
for $n$ large enough. Next the right
inequality for $\lambda_{\# S_n}(\Delta_\eta)$, (if this eigenvalue
exists) is a straightforward application of Theorem
\ref{thm-tree-vp-majoration} and Proposition
\ref{prop-majoration-radial-tree}. Finally, since the upper estimate
is strictly lower than $1$, the min-max
Theorem \ref{t:min-max} ensures the existence of a infinite number of
eigenvalue under the essential spectrum. 

We turn to $\Delta_1$. The left inequality is obtained by taking 
$ c= \sqrt{\eta(n)-1}$ in \eqref{e:psic} since $\psi_c$ can be written
as  
\[
 \psi_c(x)= \eta(n)-\left(c+\frac{\eta(n)-1}{c}\right), \quad x\in S_n.
\]
Corollaries \ref{cor-ess-1} and \ref{cor-vp-minoration-min-max} give
the desired result for the essential spectrum. Then we have:
\[\lambda_{(\# B_{n-1}+1)}\leq \lambda_{(\# S_{n})}\leq \eta(n)\]
by taking Dirac test functions. \end{proof}

\subsection{The case of  general weakly  spherically symmetric graphs}
In this section, we investigate the case of general weakly
spherically symmetric graphs. 

\begin{proposition}\label{prop-eigenvalue-gen-wss}
\begin{enumerate}
\item Let $\Gr:=(\Vr,\Er,m)$ be a weakly spherically symmetric graph
  with $m=\eta$.  Assume  that $p_+(n)(1-p_+(n))$ is non-increasing. 
Then, if the corresponding eigenvalues exist, we have:
\[
 \lambda_{(\# B_{n-1}+1)}(\Delta_\eta) \geq  1- 2 \sqrt {p_+(n)(1-p_+(n))}
\]
and
\[
\lambda_n(\Delta_\eta) \leq 1-\sqrt{ p_+(3n-2) p_-(3n-2)}.
\]

\item
Let $\Gr:=(\Vr,\Er,m)$ be a weakly spherically symmetric graph with
$m=1$.  Assume  that both $\eta(n)$ and  $\eta(n)- 2 \sqrt
{\eta_+(n)\eta_-(n)}$ are non-decreasing. 
Then we have, if the corresponding eigenvalues exist: 
\[
 \lambda_{(\# B_{n-1}+1)}(\Delta_1) \geq  \eta(n)- 2 \sqrt {\eta_+(n)\eta_-(n)}
\]
and
\[
\lambda_n(\Delta_1) \leq \eta(2n-1).
\]
\end{enumerate}
\end{proposition}
\begin{corollary}
Under the hypothesis of Proposition \ref{prop-eigenvalue-gen-wss},  if moreover $p_+(n)\to 1$ as $n\to +\infty$, then
$\sigma_{\rm ess} (\Delta_\eta)=\{1\}$ and the min-max Theorem \ref{t:min-max} implies the existence of an infinite number of eigenvalues. Thus we have 
\[
\sigma (\Delta_\eta)=\{1\} \cup  \{\lambda_i^-, \lambda_i^+,  i\geq 1\}
\]
where $(\lambda_i^-)_{i\geq1}$ and  $(\lambda_i^+)_{i\geq1}$ are 
infinite sequences of eigenvalues converging to $1$ from below and from
above respectively. 
Similarly, if $\eta(n)- 2 \sqrt {\eta_+(n)\eta_-(n)}\to +\infty$ as $n\to \infty$, then $\sigma_{\rm ess} (\Delta_1)= \emptyset$ and there is an infinite sequence of eigenvalues tending to $+\infty$.
\end{corollary}

\begin{proof}
For the first inequality, note that by hypothesis, for all $x\in B_{n-1}^c$, 
\[
p_+(x)-p_-(x)\geq 2p_+(n)-1.
 \]
 By Theorem \ref{thm-spectrum-bd} and by Corollary
\ref{cor-vp-minoration-min-max},
\begin{align*}
 \lambda_{(\# B_{n-1}+1)}(\Delta_\eta) &\geq 1-\sqrt{1-
   \left(2p_+(n)-1\right)^2} = 1- 2 \sqrt {p_+(n)(1-p_+(n))}. 
\end{align*}
For the right inequality,
let $g_n$ be the function defined on $\Vr$ by $g_n(x)=1$ if $x\in
S_n$, $g_n(x)=\alpha_n$ if $x\in S_{n+1}$ and $g_n(x)=0$ otherwise,
where $\alpha_n$ will be chosen later.  
Since for $|i-j|\geq 3$, $d_\Gr(\supp\,  g_i, \supp\,  g_j)\geq 2$, then 
\[
\lambda_n (\Delta_\eta) \leq \max_{i\in\{1,4,\dots,3n-2\} } \frac{\langle g_i,\Delta_\eta g_i
  \rangle_\eta}{\langle g_i, g_i \rangle_\eta}. 
\]
Now a computation gives
\begin{align*}
 \frac{\langle g_n,\Delta_\eta  g_n \rangle_\eta}{\langle   g_n,  g_n \rangle_\eta}
                                                 & = \frac{\eta_-(S_n) 1^2 + \eta_+(S_n) (1-\alpha_n)^2+ \eta_+(S_{n+1}) \alpha_n^2}
                                                              {\eta(S_n) +\alpha_n^2 \eta(S_{n+1})}\\
                                                 & \leq \frac{\eta(S_n) + \alpha_n^2 \eta(S_{n+1}) -2\alpha_n \eta_+(S_n)  }
                                                              {\eta(S_n) +\alpha_n^2 \eta(S_{n+1})}\\
                                                  &=1-\frac{2\alpha_n  \eta_+(S_n)}
                                                             {\eta(S_n) +\alpha_n^2 \eta(S_{n+1})}\\
                                                  &=1-\frac{ \eta_+(S_n)}
                                                             {\sqrt{\eta(S_n) \eta(S_{n+1})}}\\
                                                   &=1-\sqrt{  \frac{ \eta_+(S_n)} {\eta(S_n) }
                                                             \frac{\eta_-(S_{n+1})}{ \eta(S_{n+1})}}\\
                                                    &=1-\sqrt{ p_+(n) p_-(n)},\\
\end{align*} 
with the choice $\alpha_n=\sqrt{\frac{\eta(S_n)}{\eta(S_{n+1})}}$ and
since $\eta_+(S_{n})=\eta_-(S_{n+1})$. 
For the Laplacian $\Delta_1$, as before, we obtain
\begin{align*}
 \lambda_{(\# B_{n-1}+1)}(\Delta_1) &\geq \eta(n)- 2 \sqrt
 {\eta_+(n)(\eta(n)-\eta_+(n))}. 
\end{align*}
For the right inequality, let $g_n$ be the function defined on $\Vr$
by $g_n(x)=1$ if $x\in 
S_n$ and $g_n(x)=0$ otherwise. Since for $|i-j|\geq 2$, $d_\Gr(\supp\,
g_i, \supp\,  g_j)\geq 2$, then  
\begin{align*}
\lambda_n (\Delta_1) \leq \max_{i\in\{1,3,\dots,2n-1\} }
\frac{\langle
  g_i,\Delta_1 g_i   \rangle_1}{\langle g_i, g_i \rangle_1} &
\leq \frac{\eta_+(S_{2n-1})+\eta_-(S_{2n-1})}{|S_{2n-1}|} \leq \eta(2n-1).
\end{align*}
This ends the proof.
\end{proof}

\subsection{The case of antitrees}
A  simple graph $\Gr$ is an  \emph{antitree} if there exists a 1-dimensional decomposition $(S_n)_{n\geq 0}$ of $\Vr$ such that $\eta_+(x)=\# S_{n+1},\eta_-(x)=\#S_{n_1}$ and $\eta_0(x)=0$ for all $x\in S_n, n\geq 0$. 
Antitrees are bipartite graphs. 
The spectral decomposition of the Laplacian on antitrees is made in \cite{BK}. It is shown that the spectrum of $\Delta_\eta$ is  the union of $\{1\}$ and the spectrum of a Jacobi matrix.

This comes from the fact that if $f$ is orthogonal to radial functions then $A_m f= 0$ and then $\Delta_m f= \deg(\cdot) f$. Therefore for $\Delta_\eta$, $1$ is an eigenvalue with infinite multiplicity.  The Jacobi matrix corresponds to the action of the Laplacian on radial functions.
The upper estimate for the eigenvalues in Proposition \ref{prop-eigenvalue-gen-wss} is in fact an estimate for the eigenvalues associated to this radial part of the Laplacian. 
Therefore, in general, the upper estimate for the eigenvalues of $\Delta_m$ in Proposition \ref{prop-eigenvalue-gen-wss} is reasonable.


\section{An Allegretto-Piepenbrink type theorem for the essential
  spectrum}\label{s:6} 
In this section, we prove the reverse part of an
Allegretto-Piepenbrink type theorem for the essential spectrum. This
gives a partial reverse statement of Corollary \ref{cor-lyap-spec}. 
\begin{theorem}\label{thm-reverse}
Let $\Gr:=(\Vr,\Er,m)$ be a weighted graph. Let $\lambda^0:= \inf
\sigma(\Delta_m)$ and $\lambda^0_{ess}:=\inf \sigma_{\rm
  ess}(\Delta_m) $. Then we have:
\begin{enumerate}[{\rm a)}]
 \item Let $\lambda \leq \lambda^0$, then there exists a positive
   function $W$ such that  
\[
\tilde \Delta_m W (x) \geq  \lambda W(x).
\]
 \item For all
$\ve>0$, there exist $N_1:=N_1(\ve)\geq 1$, $C:=C(\ve)>0$, and  a positive function $W$ such 
that
\[
\tilde \Delta_m W (x) \geq  \left( \lambda^0_{ess}-\ve \right) W(x)- C \bone_{B_{N_1}}(x).
\]

 \item\label{iii:thm-reverse} If moreover: \beq \label{assumption-min-m}\inf\{ m(x),x\in \Vr\} >0, \eeq
 then for all
$\ve>0$, there exist $N_2:=N_2(\ve)\geq 1$ and  a positive function $W$ such 
that
\[
\tilde \Delta_m W (x) \geq  \left( \lambda^0_{ess}-\ve \right) \bone_{B_{N_2}^c}(x) W(x).
\]

\item If $\Gr:=\Gr_\N$ is a weighted graph on $\N$ such that  
\[m(\N)   =+\infty\]
 then the conclusion of \ref{iii:thm-reverse}) holds without the additional
  assumption \eqref{assumption-min-m}. 
\end{enumerate}
\end{theorem}

\begin{remark}
  The condition $\inf\{ m(x),x\in \Vr\} >0$ is equivalent to the
  inclusion $\ell^2(\Vr,m) \subset \ell^\infty(\Vr,m)$. 
\end{remark}

\begin{proof}
The first point is well-known, see Theorem 3.1 in \cite{HK} for a proof. 
Let us turn to the other points. Since  $\lambda^0_{ess}:=\inf \sigma_{\rm ess}(\Delta_m)$, by Persson lemma,
 there exists $K$ such that the infimum of the spectrum of the
 Dirichlet operator $\Delta_m^{B_K^c}$ is larger than $
 \lambda^0_{ess}-\ve$. 

The operator $ \left( \Delta_m^{B_K^c} -
  (\lambda^0_{ess}-\ve) \right)^{-1}$ is thus well defined on $B_K^c$ and
is positive improving.  Indeed, it is well-known that the Dirichlet
heat semigroup is positive improving, see for example Corollary 2.9
in  \cite{KL2}.  
Let $\psi$ be  a non-negative and non trivial  function in 
$\ell^2(B_K^c,m)$, since $\lambda^0_{ess}-\ve$ is strictly below
spectrum of $\Delta_m^{B_K^c}$, writing 
\[
 \left( \Delta_m^{B_K^c} - (\lambda^0_{ess}-\ve)
\right)^{-1}\psi(x) = \int_0^\infty e^{t(\lambda^0_{ess}-\ve) }P_t \psi(x)\, dt,
\]
we see that $ \left( \Delta_m^{B_K^c} - (\lambda^0_{ess}-\ve)
\right)^{-1}\psi(x)>0$, for all $x$. 
 
Now, let $\psi$ a non-negative (non trivial) function in $\Cc_c(B_K^c)$ and
consider $\phi= \left( \Delta_m^{B_K^c} - (\lambda^0_{ess}-\ve)
\right)^{-1}\psi$. 
 Then $\phi>0$ and $\phi$ satisfies $\Delta_m^{B_K^c} \phi (x) \geq
 (\lambda^0_{ess}-\ve)\phi(x)$. This gives the second point.

Now, since $\phi \in \ell^2(B_K^c,m)$, the condition $\inf\{ m(x),x \in \Vr\} >0$ implies $\phi (x) \to 0$ when $|x|\to \infty$.
Let $\ve'= \frac{1}{2} \min\{ \phi(x) ,x\in S_{K+1} \}$, then the set
$A_{\ve'}=B_K \cup \{x\in B_K^c, \phi(x)> \ve' \}$ is finite. 

Recall that $\delta (A_{\ve'}^c)$ is the set  of points $x$ in $A_{\ve'}^c$ who have
a neighbor which belongs to $A_{\ve'}$. 
Let $u$ be  the harmonic function in $A_{\ve'}\cup \delta (A_{\ve'}^c)$ such that 
\[
\left\{ \begin{array}{l}
\Delta_m u = 0 \textrm{ on } A_{\ve'}\\
 u=\phi \textrm{ on } \delta (A_{\ve'}^c)
\end{array}
\right.
\]

Define then the function  $W$ on $\Vr$ as:
\[
\left\{\begin{array}{l}
W(x) = u(x) \textrm{ on } A_{\ve'}\\
W(x)=\phi(x) \textrm{ on } A_{\ve'}^c\\
      \end{array}
\right.
\]
Clearly, $\Delta_m W(x)=0$ for   $ x\in A_{\ve'}$ and $ \Delta_m W(x) \geq (\lambda^0_{ess}-\ve)W(x)$ for  $x\in  \textrm{Int}\, (A_{\ve'}^c)$.
It remains to look at the points $x$ in $ \delta (A_{\ve'}^c)$. Let $x\in  \delta (A_{\ve'}^c)$.
For $y\in  A_{\ve'}^c$, we have $W(y)=\phi(y)$ and for $y\in \delta A_{\ve'}$, since $\delta A_{\ve'} \subset B_K^c \cap A_{\ve'}=\{x\in B_K^c, \phi(x)> \ve' \}$,  by a maximum principle for harmonic functions 
$W(y)=u(y) \leq \ve' \leq \phi(y)$, therefore
\[
\tilde \Delta_m W(x) \geq \tilde \Delta_m \phi(x) = \tilde \Delta_m^{B_K^c} \phi(x) \geq  (\lambda^0_{ess}-\ve) \phi(x)=  (\lambda^0_{ess}-\ve) W(x).
\]

Let us turn to the proof of the last point. With the same construction as above,
 there exist $N\geq 1$ and  $\phi^\N>0$ on $\lint N+1,+\infty \rint$ such that, for $n\geq N+1$, $\Delta_{\Gr_\N}^{\lint N+1,+\infty \rint} \phi^\N (n) \geq
 (\lambda^0_{ess}-\ve)\phi^\N(n)$. Since $m(\N)=+\infty$, $\phi^\N$ can not be non-decreasing.
  Therefore there exists $n_0\geq N+1$, such that $\phi^\N(n_0+1)\leq \phi^\N(n_0)$. 
We can now perform the  cut and paste procedure by taking $W$ to be the function
\[
\left\{\begin{array}{l}
W(n) = \phi^\N(n_0) \textrm{ for } n\in \lint 0,n_0 \rint \\
W(n)=\phi^\N(n) \textrm{ for  } n\geq n_0+1.\\
      \end{array}
\right.
\]
Clearly, $W$ is the  desired super-harmonic function.
\end{proof}

\begin{remark}
 Note that, in the above proof, since $\phi^\N$ can not have local minimum, if we have $\phi^\N(n+1)\leq \phi^\N(n)$ for some $n$, then $\phi_N$ is non-increasing on $\lint n,+\infty)$.
\end{remark}

\section{Harnack inequality and limiting procedures}\label{s:7}
In this section, we recall how to obtain a super-solution on the entire set of vertices $\Vr$ given  a sequence of super-solution defined on a exhaustive sequence of finite sets. 
We recall that the graph is supposed to be connected.
The results of this section  are taken from \cite{HK}. The only difference is that, here, we consider a non-negative function $\lambda$ in place of a constant. The proofs adapt 
straightforwardly  and will not be presented.

First we begin by the Harnack inequality for non-negative super-solutions.
\begin{theorem}
Let $\Wr \subset \Vr$ be a finite and connected set. Let $\lambda: \Vr\to \R$ be a non-negative function. 
There exists a constant $C_\Wr$ such that for all  non-negative function $W: \Vr \to [0,_+\infty)$ satisfying
$(\Delta  -\lambda(x)) W(x) \geq 0$ for all $x\in \Wr$, we have
\[
\max_{x\in \Wr} W(x) \leq C_\Wr \min_{x\in \Wr} W(x).
\]
\end{theorem}
As Corollary we obtain: 

\begin{corollary}
Let $\Wr \subset \Vr$ be a connected set. Let $x_0\in \Vr$ and let $\lambda: \Vr\to \R$ be a non-negative function. 
 For all $x\in \Wr$, there exists a constant $C_x:=C_x(x_0,\Wr)$ such that for all  
 non-negative $W: \Wr \to [0,+\infty)$ satisfying $W(x_0)=1$ and
$(\Delta -\lambda(x)) W(x) \geq 0$ for all $x\in \Wr$, we have
\[
C_x ^{-1} \leq W(x) \leq C_x.
\]
 \end{corollary}

\begin{remark}
 Obviously, the last corollary can be used with $\Wr=\Vr$. 
\end{remark}

We now turn to the main result of this section.

\begin{theorem}\label{global-super-sol}
Let $x_0\in \Vr$ and let $\lambda: \Vr\to \R$ be a non-negative function. Let $(\Wr_n)_n$ be an exhausting sequence of $\Vr$. Assume that there exists a sequence of non-negative functions $W_n: \Wr_n\to [0,+\infty)$ 
satisfying $W_n(x_0)=1$ and $ (\Delta _n -\lambda(x) )W_n(x)\geq 0$ 
(respectively $ (\Delta _n -\lambda(x) )W_n(x)=0$) for all $x\in \Wr_n$.  Then there exists a positive function $W:\Vr\to (0,+\infty)$ such that 
$ (\Delta _n -\lambda(x) )W(x)\geq 0$ 
(respectively $ (\Delta _n -\lambda(x) )W(x)=0$) for all $x\in \Vr$.
\end{theorem}
\section{Probabilistic representation of positive super-harmonic
  functions}\label{s:8}  

In the classical situation of Poincar\'e inequality, there is a strong
link between the linear Lyapunov functions and the hitting times of
some compact sets for a stochastic process, see \cite{CGZ}. Here we
develop an analogy of these results. In all this section, $(\Wr_n)_{n
  \geq 0}$ will denote an exhaustive sequence of $\Vr$, see Definition
\ref{d:exhau}.   
 
\subsection{Discrete and continuous time Markov chains}   
In this section, we present the Markov processes whose generator is
given by (minus) the Laplacian on the graph.  
In the case of a general weighted Laplacian,  we can associate a
continuous time Markov chain. 
In the case of the normalized Laplacian, we can associate both  a
continuous time and  a discrete time Markov chain. 
More details about the construction and the properties of these Markov
process can be found in the monograph \cite{Nor}.

\subsubsection{The discrete time Markov chain associated to $\Delta_\eta$}\label{s:random-walk}
We begin by the simplest case of the normalized Laplacian. Consider
the Markov chain $(X_k^{x_0})_{k\geq 1}$ starting in $x_0$ on the
graph whose transition probabilities are given 
by  
\[p(x,y):= \frac{\Er(x,y)}{\eta(x)},\]
for all $x,y\in \Vr$. Then, set $Pf(x) =\sum_y p(x,y) f(y)$ for all
$f\in \ell^\infty(\Vr)$. For $k\geq 0$ and $x \in \Vr$, one has 
\[
 P^kf(x) =\E\left[ f(X_k^x) \right].
\] 
The generator
 of the above discrete  time Markov chain
random walk is given by $P-\Id$ and then equals  $-\Delta_\eta$; that is for $f\in \ell^\infty(\Vr)$,
\[
 \E\left[ f(X_1^x) \right] -f(x) = -\Delta_\eta f(x), \quad x\in \Vr.
\]
The measure $\eta$ satisfies $\eta(x) p(x,y) =\eta(y) p(y,x)$ for $x,y \in \Vr$. It is
 symmetric 
 (and hence
invariant) for the Markov chain. 

\subsubsection{The continuous time Markov chain associated to $\Delta_m$}\label{s:markov-process}
Now we turn to the general case.
With the above notation, the Laplacian $-\Delta_m$ can be written as
\[
\Delta_m f (x) = \deg(x) \sum_y p(x,y) (f(x)-f(y)). 
\]

We construct here the \emph{minimal} right continuous Markov chain
$(X_t)_{t\geq 0}$ associated to $-\Delta_m$. It corresponds to the
process  killed at infinity. We denote by $e(X)$ its explosion time
(recall that $X$ depends on the choice of the initial law). 
We recall two useful constructions of the continuous time Markov
chain   when the initial law is $\delta_{x_0}$.  We denote it by
$(X_t^{x_0})_{t\geq 0}$.

First we can construct $(X_t^{x_0})_{t\geq 0}$ as follows: At time
$t=0$, $X_0^{x_0}=x_0$. It stays in $x_0$ during an exponential random
time of parameter $\deg(x_0)$ and then jumps in a point
$y$ chosen with probability $p(x_0,y)$. We  then iterate this procedure.

Another useful equivalent construction of the process $(X_t^{x_0})_{t\geq
  0}$ is the following. At time $t=0$, $X_0^{x_0}:=x_0$. For each,
neighbor $y$ of $x_0$, we let $E_y$ be an independent exponential
random clock variable of parameter $\deg(x_0,y)$.  
Consider $T:=\min \{E_y, y\sim x_0\}$. Let $z$ be the neighbor of
$x_0$ such that $E_z=\min \{E_y, y\sim x_0\}$, $z$ is unique almost
surely.  
We set $X_t:=x_0$ for $0\leq t<T$, $X_T:=z$ and repeat this construction.

Using using the memorylessness property of the exponential
distribution and Lemma \ref{lemma-exp1} below, it is easy that both
constructions are equivalent and that   $(X_t^x)_{t\geq 0}$ is a
Markov process. Moreover, the jump chain associated to $(X_t^x)_{t\geq
  0}$ is the discrete time Markov chain of generator $-\Delta_\eta$. 

\begin{lemma}\label{lemma-exp1}
 Let $(E_i)_{1\leq i\leq n}$ be $n$ independent exponential random  variables of parameter $c_i >0$, then the variable  
$\min \{E_i,1\leq i\leq n\}$ is also an exponential  random variable
of parameter $c_1+\dots + c_n$.  
Moreover, for all $1\leq r \leq n$ we have:
\[\PR \left(\min \{E_i,1\leq i\leq n\}= E_r\right)= \frac{c_r}{c_1+\dots +c_n}.\] 
\end{lemma}

The next lemma concerns also the memorylessness property of the exponential distribution. It will be useful to add some ``artificial jumps'' in the construction of the process $(X_t)_{t\geq 0}$.

\begin{lemma}\label{lemma-exp3} Let $n\geq 1$ and $c_1,\dots,c_n>0$.
 Let $(E_{i,j})_{i\geq1, 1\leq j \leq n}$ be  independent exponential
 random  variables such that the   parameter of $E_{i,j}$ is
 $c_j$. Let $(A_i)_{i\geq1}$ be  independent  random  variables such
 that almost surely  
\begin{equation}\label{cond-A-infty}A_i> 0 \textrm{ and } \sum_{i=1}^\infty A_i =+\infty
\end{equation}
Let $k$ be defined by 
\[
k:=\inf\{i\geq 1, \min(E_{i,1},\dots,E_{i,n},A_i)\neq A_i\}
\]
Then $k$ is finite almost surely and the random variable 
$B:=A_1+\dots+A_{k-1}+ 
\min(E_{k,1},\dots,E_{k,n})$ is also an exponential
random  variable of parameter $c:=c_1+\dots+c_n$. 
Moreover for all $1\leq r\leq n$, we have:
\begin{align*}
\PR \left( \min(E_{k,1},\dots,E_{k,n})=E_{k,r} \right) &=  \PR \left(
  \min(E_{1,1},\dots,E_{1,n})=E_{1,r} \right)
= \frac{c_r}{c_1+\dots +c_n}.
\end{align*}
\end{lemma}

The Lemma \ref{lemma-exp3} allows us to add some ``artificial jumps'' in the construction of the Markov process $(X_t)_{t\geq 0}$. Indeed, it implies  that we can also construct $(X_t)$ as follows:
If at time
$t$, $X_t=x$, then as before for each neighbor $y$ of $x$,  we let $E_y$ be an independent exponential random clock variable of parameter $\deg(x,y)$. We let also $E_x$ be another independent exponential random clock variable. Let $\deg(x,x)$ be its parameter.
Consider $\tilde T:=\min \{E_y, y\sim x \textrm{ or } y=x\}$. Let $z$ be the  unique vertex such that $E_z=\min \{E_y, y\sim x \textrm{ or } y=x\}$.
We set $X_s:=x$ for $t\leq s<t+\tilde T$, $X_{\tilde T}:=z$ and repeat this construction.
Moreover at each step, the choice of the parameter $\deg(x,x)$ can change (with the restriction that it has to satisfy condition \eqref{cond-A-infty}).
The only difference with the previous construction is that $\tilde T$ does not really correspond anymore to a physical jump of the process.

This modification of the construction will be useful in the coupling
arguments of section \ref{s:10}.

In the above constructions, the sequence of the (random) times of the  jumps of the Markov process $X$ is increasing, thus has a limit in $(0,\infty]$.
This limit is called the \emph{explosion time} of  Markov process $X$ and is denoted by
$e(X)$.

We can now associate a \emph{continuous time semigroup} $P_t$ for $f\in
\ell^\infty(\Vr)$ by 
\[
 P_tf(x):= \E\left [ f(X_t^x) \bone_{\{t<e(X^x)\} }\right], \quad t\geq 0.
\]
For $t$ small,  using exponential distributions, it is easy and well-known to compute  explicitly the first order expansion of  law of $X_t^x$. One gets
\[\begin{array}{l}
 \PR(X_t^x =x) = 1- \deg (x) t + o(t)\\
 \PR(X_t^x =y) = \deg (x) p(x,y) t +o(t), \textrm{ for } y\neq x.
 \end{array}
\]
In particular for $f\in \ell^\infty(\Vr)$,  the following    pointwise convergences hold: 
\[
 \lim_{t\to 0^+}  P_t f (x) = f (x)  
\]
and 
\[
 \lim_{t\to 0^+} \frac{P_t f(x)-f(x)}{t}= - \tilde \Delta_{\Gr,m} f(x).
\]

Actually, in the sequel, we only need to consider the above  Markov process stopped outside a finite set.
Let $B $ be a finite subset of $\Vr$ and let 
\[T_{B^c}:= \inf\{t\geq 0, X_t \in B^c\}\]
  the
\emph{hitting time} of the set $B^c$ for the continuous time Markov
chain   $(X_t)_{t\geq0}$. Clearly, since each connected component of $\Gr$ is infinite and $B$ is finite, by classical result on transience, $T_{B^c}$ is almost surely finite. Moreover, we also have
$T_{B^c}<e(X)$.
We can now define a new continuous time  semigroup $P_t^{D_B}$ for $f:\Vr \to \R$ by
\[
 P_t^{D_B}f(x):= \E\left [ f(X_{t\wedge T_{B^c}}^x)\right], \quad t\geq
 0, \quad x\in \Vr.
\]
A computation similar to the above shows that  for  $f:\Vr \to \R$ and $ x\in \Vr$, pointwise,  
 \[
 \lim_{t\to 0^+} \frac{P_t^{D_B} f(x)-f(x)}{t}=\left\{ \begin{array}{ccl}
                                               - \tilde \Delta_{\Gr,m} f(x) &\textrm{ if } & x\in B\\
                                               0 &\textrm{ if } & x\in B^c.\\
                                             \end{array}\right.
\]
It is not symmetric on $\Cc_c(\Vr)$. Indeed its generator can be written as
$-\Pi_B \Delta$ where $\Pi_B$ is the projection defined by $\Pi_B
f(x):= f(x) \bone_B(x)$, for all $f:\Vr \to \R$.

For $f\in \ell^\infty(\Vr)$ we could also define the semigroup: 
\[
 P_t^{D'_B}f(x):= \E\left [ f(X_t^x) \bone_ {\{t< T_{B^c}\}} \right], \quad t\geq 0.
\]
It corresponds to the usual Dirichlet semigroup. As before, one can
compute that its generator is: $-\Pi_B \Delta \Pi_B$. More precisely,
for all $f:\Vr \to \R$, pointwise , one has: 
\[
 \lim_{t\to 0^+} \frac{P_t^{D'_B} f(x)-f(x)}{t}=  -\Pi_B \Delta \Pi_B
 f(x).                                           
\]

\subsection{The normalized Laplacian in the discrete time setting} 
For simplicity, we begin with the case of 
 the normalized Laplacian $\Delta_\eta$.
Actually, Theorem \ref{hitting2} below can also be seen as a corollary
of the general Theorem \ref{hitting-cont}. A direct proof is included
for the reader more familiar with Markov chains than continuous time
Markov chains. 

\begin{theorem} \label{hitting2} 
Let $\lambda:\Vr \to [0,1)$ and let $\Lambda(x):=\frac{1}{1-\lambda(x)}$.
The  following assertions are equivalent:
\begin{enumerate}[{\rm (i)}]
\item\label{global-egal} There exists a  positive function $W$ on
  $\Vr$ such that $\tilde \Delta_\eta W (x)=  \lambda(x) W(x)$ for all
  $x \in \Vr$.  
\item \label{global-geq} There exists a  positive function $W$ on
  $\Vr$ such that $\tilde \Delta_\eta W (x)\geq  \lambda(x) W(x)$ for
  all  $x \in \Vr$.  
\item \label{family} There exists a family of positive functions $W_N$
  on $\Vr$ such that $\tilde \Delta_\eta W_N (x)\geq  \lambda(x)
  W_N(x)$ for all $ x\in \Wr_N$. 
\item \label{rep-proba} For all $N\geq 1$ and all $x\in \Vr$, we have 
\begin{equation}\label{eq-rep-proba}
 \E\left[\prod_{k=0}^{T_N-1} \Lambda(X_k^x)\right]  < + \infty,
\end{equation}
where $T_N:=T_N^x:= \inf \{n \geq 0, X_n^x \in \Wr_N^c \}$ is the hitting
time of the set $\Wr_N$  for $X_n^x$ the Markov chain on $\Gr$ starting
in $x$ and whose 
generator is $- \Delta_\eta$. 
\end{enumerate}
\end{theorem}
When the function $\lambda$ is constant, one has
$\Lambda:=\frac{1}{1-\lambda}$ and the function in \eqref{rep-proba} reads also 
$ \E_x[\Lambda^{ T_N}]:= \E[\Lambda^{ T_N^x}]$.

\begin{remark}
In this situation, one can not have $\lambda(x)\geq 1$ for some $x$
since for all positive function $W$ and all $x\in \Vr$, 
\[\tilde \Delta_\eta W (x) = \sum_{y} p(x,y)(W(x)- W(y))<\sum_{y}
p(x,y)W(x)=W(x).
\]
\end{remark}

\begin{proof}
Clearly,  \eqref{global-egal} implies \eqref{global-geq}.
The equivalence between  \eqref{global-geq} and \eqref{family} is
given by Theorem \ref{global-super-sol}. 
We now  show that \eqref{family} implies \eqref{rep-proba}. Set 
\[A_n:= \prod_{k=0}^{n-1} \Lambda(X_k^x), \mbox{ for } n\geq 1\]
and $A_0:=1$. Let $N\geq 1$, $x\in \Wr_N$ and $n\geq 0$, we have
\begin{align}\label{e:hitting2}
 \E_x \left[A_{n \wedge T_N} \right]\leq \frac{1} {\min\{W_N(x),x\in
   \Wr_N \}} \E_x \left[A_{n \wedge T_N} W_N (X_{n \wedge T_N
   }^x)\right]. 
\end{align}
Using the Abel transform $u_n v_n = u_0 v_0 + \sum_{k=0}^{n-1}
\left((u_{k+1}-u_k) v_{k+1} + u_k (v_{k+1}-v_k)\right)$, we get 
\begin{align*}
\E_x \left[A_{n \wedge T_N} W_N (X_{n \wedge T_N }^x)\right]
 =& \, W_N(x)+ \sum_{k=0}^{(n\wedge T_N) -1} \E_x \left[   A_{k+1}  \left( W_N (X_{k+1}^x) -W_N(X_k^x) \right)  \right]\\
& \,  + \sum_{k=0}^{(n\wedge T_N) -1}  \E_x \left[ (A_{k+1}-A_k) W_N(X_k^x) \right].\\
\end{align*}
The event  $A_{k+1}$ is measurable with respect to the $\sigma$-algebra $\sigma(X_1^x, \dots, X_k^x)$, thus
\begin{align*}
\E_x \left[     A_{k+1}  \left( W_N (X_{k+1}^x) -W_N(X_k^x) \right) \right]
&=
\\
&\hspace*{-2cm}= \, \E_x \left[  A_{k+1}  \E\left[  W_N (X_{k+1}^x) -W_N(X_k^x)  |
   \sigma(X_1^x, \dots, X_k^x) \right] \right]
\\
&\hspace*{-2cm}=\,  -\E_x \left[  A_{k+1}  \tilde\Delta_\eta W_N(X_k^x)   \right]
\end{align*}
and therefore
\begin{align*}
\E_x \left[A_{n \wedge T_N} W_N (X_{n \wedge T_N }^x)\right] 
&= W_N(x) + \sum_{k=0}^{(n\wedge T_N)-1}   \E_x \left[  - A_{k+1}  \tilde \Delta_\eta W_N (X_k^x)\right]  \\
 &\quad +\sum_{k=0}^{(n\wedge T_N)-1}   \E_x \left[ A_{k+1} \left(1-\frac{1}{\Lambda
       (X_k^x)}\right) W_N (X_k^x) \right]\\ 
 &\leq  W_N(x), 
\end{align*}
where  we have used $\tilde \Delta_\eta W_N(x) \geq \lambda (x)
W_N(x)=  \left(1-\frac{1}{\Lambda(x)} \right)W_N(x)$ for $x \in
\Wr_{N}$. Finally, since $T_N$ is almost surely finite and recalling 
\eqref{e:hitting2}, by letting $n\to \infty$ we obtain  \eqref{rep-proba}.

We turn to  \eqref{rep-proba} implies \eqref{global-egal}.
Set $U_N(x):= \E_x[ A_{T_N}]$ for $N\geq 1$. By hypothesis, it is
finite for  
all $N$ and all $x\in \Vr$. Let $x\in \Wr_{N}$, by Markov property, we have 
\[
U_N(x)=  \sum_{y\sim x} p(x,y) \Lambda(x) \E_y[A_{T_N^y}] = \Lambda(x)  \sum_{y,y\sim x} p(x,y)  U_N(y),
\]
thus
\begin{align*}
 \Delta_\eta U_N(x)&=  U_N(x) - \sum_{y,y\sim x} p(x,y) U_N(y)\\
           &= \left(1-\frac{1}{\Lambda (x)}\right)U_N(x) =\lambda(x) U_N (x).  
\end{align*}
Theorem \ref{global-super-sol} ends the proof.
\end{proof}

\subsection{The general case: the continuous time setting}
In the case of  a general weight $m$, we obtain the analogous of 
Theorem \ref{hitting2} for the continuous time Markov process associated
to $\Delta_m$. 
\begin{theorem}\label{hitting-cont}
Let $\lambda$ a non-negative function on $\Vr$.
 The following assertions are equivalent.
\begin{enumerate}[{\rm (i)}] 
\item \label{SHG-egal}  There exists a positive function $W$ such
  that: $\tilde \Delta_m W(x)= \lambda (x) W(x)$ for all $x\in \Vr$. 
 \item \label{SHG-geq}  There exists a positive function $W$ such
   that: $\tilde \Delta_m W(x)\geq \lambda (x) W(x)$ for all $x\in
   \Vr$. 
 \item  \label{SHN} There exists a family of positive  function $W_N$
   such that: $\tilde \Delta_m W_N(x)\geq \lambda (x) W_N(x)$ for all
   $x\in \Wr_N$. 
 \item  \label{SHE} For all $N\geq 1$ and all $x\in \Vr$, the positive function 
\begin{equation}\label{eq-SHE} 
U_N(x):=\E \left[\exp \left( \int_0^{T_N}  \lambda(X_s^x) ds \right)\right]
\end{equation}
is finite where $T_N:= \inf\{t\geq 0, X_t^x \in \Wr_N^c\}$ is the
hitting time of the set $\Wr_N^c$ for the continuous time Markov
chain  $(X_t^x)_{t\geq0}$ starting in $x$ and whose generator is
$-\Delta_m$. 
\end{enumerate}

\end{theorem}

\begin{proof}
We focus on the implications:  \eqref{SHN} implies
\eqref{SHE} and \eqref{SHE} implies \eqref{SHG-egal}. 
We start with \eqref{SHN} implies \eqref{SHE}. Let 
\[A_t(x):=\E \left[\exp \left( \int_0^{t}  \lambda(X_s^x)
    ds \right)\right].\] We have: 
\begin{align*}
 A_{t\wedge T_N}(x)  \leq  \frac{1}{\min\{W_N(z),z \in \Wr_N\}} \E
 \left[\exp \left( \int_0^{t\wedge T_N }  \lambda(X_s^x) ds \right)
   W_N(X_{t\wedge T_N}^x) \right]. 
\end{align*}
By the Dynkin  formula we get:
\begin{align*}
 & \E \left[\exp \left( \int_0^{t \wedge T_N}  \lambda(X_s^x) ds
   \right) W_N(X_{t\wedge T_N}^x) \right] \\ 
= & \, W_N(x) + \E \left[ \int_0^{t\wedge T_N}  \exp \left( \int_0^{u
    }  \lambda(X_v^x) dv \right) \big( \lambda (X_u^x) W_N(X_{u}^x) - \tilde 
  \Delta_m W_N(X_u^x)  \big) du \right]\\ 
\leq &\,  W_N(x) < + \infty
\end{align*}
since by hypothesis
$ \tilde \Delta_m W(X_u^x) -\lambda (X_u^x) W(X_u^x) \geq 0$. As $T_N$ is finite almost surely, letting $t\to \infty$ gives \eqref{SHE}.

Finally we assume \eqref{SHE}.
Using the strong Markov property, one has, for $x\in \Wr_N$ and $0<h\leq 1$,
\begin{align*}
 P_h^{D_{\Wr_N}} U_N(x)&= \E\left[ U_N\left(X_{h\wedge T_N^x}^x\right)\right]\\	
                 &= \E \left[\exp \left( \int_0^{ T_N^{X_h^x} }  \lambda \left(X_s^{X_h^x}\right) ds \right) \bone_{\{ h< T_N^x\}} +  \bone_{\{ h\geq T_N^x\}} \right]\\
                 &= \E \left[\exp \left( \int_{h\wedge T_N^x}^{ T_N^x }  \lambda \left(X_u^{x}\right) du \right) \right].
\end{align*}
Therefore, by dominated convergence theorem, since $\lambda$ is bounded on $\Wr_N$,
\begin{align*}
&\hspace{-1cm} \frac{P_h^{D_{\Wr_N}} U_N(x) - U_N(x)}{h}
\\
 &= \E \left[\exp \left( \int_{0}^{ T_N^x }  \lambda \left(X_u^{x}\right) du \right)
 \left( \frac{ \exp \left(- \int_{0}^{ h\wedge T_N^x }  \lambda \left(X_u^{x}\right) du \right) -1}{h} \right)\right]\\
 &\to - \lambda(x) U_N(x), \, \textrm{ as } h\to 0^+.
\end{align*}
But the above limit was already compute to be $- \tilde \Delta_m U_N(x)$; thus
\[
\tilde \Delta_m U_N(x)= \lambda(x) U_N(x), \textrm{ for } x \in \Wr_{N}.  
\]
Theorem \ref{global-super-sol} implies \eqref{SHG-egal}.
\end{proof}

\begin{remark}
 For the normalized Laplacian $\Delta_\eta$, both
 quantities \eqref{eq-rep-proba} 
and \eqref{eq-SHE} coincide. 
Indeed, with the above  notation, if $Z$ is  an exponential random variable of parameter $1$  and if  $0\leq \lambda<1$, then 
\[
 \E\left[\exp(\lambda Z )\right]=\frac{1}{1-\lambda}.
\]
\end{remark}

\section{Weakly spherically symmetric graphs}\label{s:9}
In this section, we assume that the graph $\Gr$ is weakly spherically
symmetric with respect to a $1$-dimensional decomposition
$(S_n)_{n\in \N}$, see Definition \ref{def-wss}.  
We prove that the bottom of the spectrum and the bottom of the
essential spectrum are the same as that of a $1$-dimensional 
Laplacian. The key point behind this result is that on a weakly
spherically symmetric graph, the radial part of the Markov process
associated to the Laplacian on $\Gr$ is still a Markov process. 
We finally construct  more explicitly the global super-solution of
Theorem \ref{global-super-sol}.  

First in the next lemma, we collect some useful known results for
weakly spherically symmetric graphs. 
\begin{lemma}\label{1dim}
 Let $\Gr$ be a weakly spherically symmetric graph  with
 respect to a $1$-dimensional decomposition 
$(S_n)_{n\in \N}$ and let $\lambda$ a be radial function on $\Vr$. The
following assertions hold. 
\begin{enumerate}[{\rm a)}]
\item For $n\geq 0$,
\[
m(S_n)\deg_+(n)=m(S_{n+1}) \deg_-(n+1),
\]
where $\deg_{a}(n):=\deg_{a}(x)$, where $x\in S_n$ and $a\in \{-,0,+\}$.
\item Given $f: \Vr \to \C$ we define $\tilde M$ to be the \emph{averaging
    operator} by: 
\[ \tilde M f(x):=\frac{1}{m(S_n)}
  \sum_{\tilde x \in S_n} f(\tilde x) m(\tilde x) , \quad x\in S_n.\]
We have the following algebraic commutation 
\begin{align*}
\tilde\Delta_\Gr \tilde M f=\tilde M \tilde\Delta_\Gr f.
\end{align*}  
\item If there exists a positive function $W$ which satisfies 
\begin{align}\label{e:1dim}
\tilde \Delta_\Gr W(x)=\lambda(x) W(x), \mbox{ for all }
x\in \Vr,
\end{align}
 then there also exists a positive  radial function which satisfies
 \eqref{e:1dim}.
\item Assuming that $\deg_+(n)\neq 0$ for all $n\in \N$, then  the
  vector space of radial functions $W$ which satisfy the algebraic 
  relations \eqref{e:1dim}
is a $1$-dimensional vector space.

\item Moreover, if $W$ is a radial function which satisfies \eqref{e:1dim} and if  both  $\lambda$ and  $W$ are non-negative 
  on $\Vr$, then  $W$ is a
  non-increasing radial function. 
\end{enumerate} 
\end{lemma}

\begin{proof}
The point  a)  is a direct consequence of the relation: 
\begin{align*}
 \sum_{x\in S_n} m(x) \deg_+(x)&= \sum_{x\in S_n} \sum_{y\in S_{n+1}} \Er(x,y)\\
                               &=  \sum_{y\in S_{n+1}} \sum_{x\in S_n}
                               \Er(y,x)= \sum_{y\in S_{n+1}} m(y)
                               \deg_-(y) 
\end{align*}
and the definition of weakly spherically graphs.
 b) and c) were already proven in Lemma 3.2 in \cite{KLW} and Lemma
 3.2.1 in \cite{Woj}, respectively. 
 Let now $W$ be a radial function; $W$ satisfies \eqref{e:1dim} if and
 only it satisfies  
 \[
  \left\{ \begin{array}{l}
           \deg_+(0) W(1)= (\deg_+(0) - \lambda(0)) W(0)\\
            \deg_+(n) W(n+1)= (\deg_+(n)+ \deg_-(n) - \lambda(n)) W(n) - \deg_-(n) W(n-1)   
          \end{array} \right.
 \]
for $n\geq 1$. Thus $W$ is determined by its value in 0. This gives
d). If moreover $W$ and $\lambda$ are non-negative, one has that
$W(1)\leq W(0)$ and writing 
\[
  \deg_+(n) (W(n+1)-W(n) )=  \deg_-(n)( W(n)-W(n-1) - \lambda(n)) W(n) 
\]
for $n\geq 1$, by immediate induction,  e) holds.
\end{proof}

We now study the radial part of the Markov process associated to a
weakly spherically symmetric graph. 
\begin{proposition}\label{prop-m-1d}
Let  $\Gr:=(\Vr, \Er, m)$ be a weighted graph and let $(S_n)_{n\geq
  0}$ be a  $1$-dimensional decomposition of $\Gr$.  Let $(X_t)_{t\geq 0}$  be the 
minimal continuous time Markov chain on $\Vr$ associated to
$-\Delta_{\Gr}$ (see section \ref{s:markov-process}). 
Then  the graph $\Gr$ is weakly spherically symmetric
with respect to  $(S_n)_{n\geq 0}$
if and only if the process  $(|X_t|)_{t\geq 0}$ is a continuous time
Markov chain on $\N$.

Moreover in this case, the generator $L^\N$ of the Markov process $(|X_t|)_{t\geq 0}$ is given by the formula
 \begin{align}\label{e:GenN}
L^\N f (n)= \deg_+(n) (f(n+1)-f(n)) + \deg_-(n)(f(n-1)-f(n))
\end{align} 
for  $f\in \ell^\infty(\N)$.
   It corresponds exactly to $- \tilde \Delta_{\Gr_\N}$ where 
$\Gr_{\mathbb N}:=(\N, \Er_\N,
m_\N)$ with
\begin{align}\label{e:graph-N}
\left.\begin{array}{rl}
\Er_\N(n,m):=&\hspace*{-0.3cm}   \left\{\begin{array}{cl}
                  m(S_n)\deg_\pm(n), & \mbox{ when } m =n\pm 1
                  \\
                 0, & \mbox{ otherwise}
                \end{array}\right. 
 \\
m_\N(n):=&\hspace*{-0.2cm} m(S_n),
\end{array}\right. 
\end{align}
for all $n,m\in \N$. 
\end{proposition}

Note that $\Er_\N$ is symmetric by Lemma \ref{1dim} a). 

\begin{proof}
First we assume that $\Gr$ is weakly spherically symmetric.
We provide an explicit construction of the process $(X_t)_{t\geq 0}$.
The desired properties for the process
$(|X_t|)_{t\geq 0}$ will follow. 
Let $x:=X_0$, $T:=0$ and $k:=0$.
We begin to describe the iteration procedure:

1) We let run three independent (and independent of all the possible
previous steps) random  exponential clock variables
$E_+(|x|),$ $E_0(x),$ $E_-(|x|)$ of parameter  $\deg_+(|x|),$ $\deg_0(x),$
$\deg_-(|x|)$, 
respectively. We then replace $T$ by the time  given by
$T+\min(E_+(|x|),E_0(x),E_-(|x|))$. 

2) If the above minimum equals $E_+(|x|)$ or $E_-(|x|)$, we set $T_{k+1}:=T$ and replace $k$ by $k+1$. We let the process $X$ stay in $x$ until the time $(T_{k+1})^-$ and   jump at  time $T_{k+1}$ in a point $z
\in S_{|x|\pm 1}$ whether the minimum equals
$E_\pm(|x|)$.  We then  go to 3).   
 
If the above minimum equals $E_0(x)$, we let the process $X$ stay in $x$ until the time $T^-$ and  jump in a point $\tilde x
\in S_ {|x|}$ at this time $T$ and repeat 1) with $x$ replaced by
$\tilde x$. 

3) Replace $x$ by $z$ and repeat 1).

With the above construction, the sequence $(T_k)_{k\geq 0}$
corresponds exactly to the sequence of times of the  jumps  associated to
the process $(|X_t|)_{t\geq 0}$ in $\N$. 

By Lemma \ref{lemma-exp3}, each time $T_k$ in the algorithm is
almost surely finite. Indeed for each $n\geq 0$, since $S_n$ is
finite, $\sup_{x\in S_n} \deg_0(x)<\infty$; this ensures that  
hypothesis \eqref{cond-A-infty} is satisfied.   
Moreover $T_{k+1}- T_k$ corresponds to the minimum of two independent
random exponential variables $Z_+,Z_-$ of parameter
$\deg_+(|X_{T_k}|),\deg_-(|X_{T_k}|)$, respectively. 

It is then  clear that the process $(|X_t|)_{t\geq 0}$ is a continuous
time Markov chain whose generator is given by \eqref{e:GenN}. 

Now assume  $(|X_t|)_{t\geq 0}$ is  a continuous time Markov chain on
$\N$. Since $(|X_t|)_{t\geq 0}$ can only make jumps of size $1$, the
generator $L^\N$ reads
\[
 L^\N f (n)= \alpha_+(n) (f(n+1)-f(n)) + \alpha_-(n)(f(n-1)-f(n))
\] 
for $f\in \ell^\infty(\N)$ and some constants $\alpha_\pm (n)\geq 0, n\in\N$
(and $\alpha_-(0)=0$). 
Let $P_t$ and $P_t^\N$ the semigroup on associated to $(X_t)_{t\geq
  0}$ and $(|X_t|)_{t\geq 0}$, respectively. 
Let $x\in \Vr$ and set $n:= |x|$. Consider $f:=\bone_{n+1}\in
\Cc_c(\N)$ and  $g:=f \circ |\cdot | \in \Cc_c(\Vr)$, one has 
\[
 P_t (g) (x) = \E \left[g(X_t^x)\right]= \E \left[f(|X_t^x|)\right]=
 P_t^\N (f)(n). 
\]
Taking derivative at $t=0^+$ gives 
\[
\deg_+(x) = -\Delta_\Gr g (x)= L^\N f(n)= \alpha_+(n).
\]
This shows that for $x\in \Vr$, the quantity $\deg_+(x)$ depends only on $|x|$. 
Similarly, one has that $\deg_-(x)$ depends also only on $|x|$; that
is $\Gr$ is weakly spherically symmetric.  

\end{proof}

\begin{remark}
It is a remarkable fact that the quantity $\deg_0$ does not appear in
the generator of the process $(|X_t|)_{t\geq 0}$ on a weakly spherically
symmetric graph. 
\end{remark}

We now show that the bottom of the spectrum and the essential spectrum
for the two Laplacians coincide.

\begin{theorem}\label{cor-inf-sym}
Let $\Gr:=(\Vr, \Er, m)$ be a weakly symmetric weighted graph such that
$m(\Vr)=+\infty$. With the above notation, we have  
 \[\inf \sigma(\Delta_{\Gr}) =  \inf \sigma(\Delta_{\Gr_\N})
\textrm{ and }
\inf \sigma_{\rm ess} (\Delta_{\Gr})= \inf \sigma_{\rm ess} (\Delta_{\Gr_\N}).
\]                      
\end{theorem}

\begin{proof}
We start with a general fact. Given $f:\N\to \C$, let $g:\Vr\to \C$ be
defined by   $g(x):=f(|x|)$, then for $x\in S_n$, observe that
\[
 \tilde \Delta_{\Gr} g(x)= \tilde \Delta_{\Gr_\N} f(n) \textrm{ and } \Vert
 g\Vert_{\ell^2(\Gr, m)}= \Vert f\Vert_{\ell^2(\N,m_{\N})}. 
\]
It follows easily that $\Dc(\Delta_{\Gr})\cap (\C^\Vr)_{\rm
  rad}=\Dc(\Delta_{\Gr_\Nr})$, where $(\C^\Vr)_{\rm
  rad}$ denotes the set of radial (w.r.t.\ the $1$-dimensional
decomposition) functions $f:\Vr\to \C$. 
It easily follows that  $\sigma (\Delta_{\Gr})\subset \sigma
(\Delta_{\Gr_\N})$ and $\sigma_{\rm ess} (\Delta_{\Gr}) \subset \sigma_{\rm ess}
(\Delta_{\Gr_\N})$.
 
For the reverse inequality, we do the proof only for the bottom of the
essential spectrum. The proof for the bottom of the spectrum is
similar and uses the first point of Theorem \ref{thm-reverse}. 
 Let $\lambda_{\N,\rm ess}^0:=\inf \sigma_{\rm ess} (\Delta_{\Gr_\N})$.
By the third point of Theorem \ref{thm-reverse}, 
for all $\ve >0$, there exist $n_0:=n_0(\ve)$ and  a positive function
$W$ on $\N$ such that   
\[
\tilde \Delta_{\Gr_\N} W(n)\geq (\lambda_{\N,\rm ess}^0-\ve) \bone_{n\geq n_0} W(n).
\]
Let $U:\Gr\to (0,\infty)$ be the function defined by $U(x):=W(|x|)$. Since $\Gr$ is a weakly symmetric graph, for $x\in S_n$, we have
$ \tilde \Delta_{\Gr} U(x)=\tilde \Delta_{\Gr_\N} W(n)$. Therefore,
\[
\tilde \Delta_{\Gr} U(x)\geq (\lambda_{\N,\rm ess}^0-\ve)\bone_{|x|\geq n_0} U(x).
\]
Finally Corollary \ref{cor-lyap-spec} and letting $\ve\to 0$ gives 
\end{proof}

\begin{remark}
The above proof also shows that, in the case of weakly spherically
symmetric  graphs,  the conclusion $(c)$ of Theorem \ref{thm-reverse} 
also holds without the additional hypothesis \eqref{assumption-min-m}. 
\end{remark}

We turn to the the case of the normalized Laplacian. Note that for 
weakly spherically symmetric graphs, since $\deg\equiv 1$ and since
$\deg_{\pm}$ are radial, $\deg_0$ is also radial.  
The next proposition is the discrete
analogous of Proposition \ref{prop-m-1d}. The proof is
straightforward. We omit it. 

\begin{proposition}\label{prop-bd-1d}
Let  $\Gr:=(\Vr, \Er)$ be a  graph and let $(S_n)_{n\geq
  0}$ be a  $1$-dimensional decomposition of $\Gr$.  Let
$(X_k)_{k\in\N}$  be the  discrete time Markov chain on $\Vr$ whose
generator  is $-\Delta_{\Gr,\eta}$ (as defined in section
\ref{s:random-walk}).  
Then  the graph $(\Vr, \Er, \eta)$ is weakly spherically symmetric
 with respect to  $(S_n)_{n\geq 0}$
if and only if the process  $(|X_k|)_{k\in\N }$ is a Markov chain on $\N$.

Moreover, in this case, the transition probabilities  of the  Markov
chain  $(|X_k|)_{k\in \N}$ on $\mathbb N$ are given by 
\[
\left\{
\begin{array}{lcl}
 p(n,n+1)&:=&p_+(n)\\
p(n,n-1)&:=&p_-(n),\\
p(n,n)&:=& p_0(n).
\end{array}
\right. 
\] 
The generator of the  Markov chain  $(|X_k|)_{k\in \N}$ corresponds
exactly to $- \tilde \Delta_{\Gr_\N}$ where   $\Gr_\N:= (\N, \Er_\N,
m_\N)$ with 
\begin{align*}
\Er_\N(n,m)&:=\left\{\begin{array}{ll}
m(S_n){p}_\pm(n), & \mbox{ when } m =n\pm 1,
\\
0 & \mbox{ otherwise.}
\end{array}\right. 
 \\
m_\N(n)&:=m(S_n).
\end{align*}
For $f\in \Cc_c(\N)$,  $-\Delta_{\Gr_\N}$ can be written by as ,
\begin{align}\label{e:GenN'}
-\Delta_{\Gr_\N} f (n)= p_+(n) (f(n+1)-f(n)) + p_-(n)(f(n-1)-f(n)). 
\end{align}
\end{proposition}

We go back to the general setting and provide a more explicit construction
of the super-harmonic function of Theorem \ref{hitting-cont}. 
\begin{proposition}\label{prop-lafctW-S0}
Assume the graph $\Gr:=(\Vr, \Er, m)$ is weakly spherically symmetric
with respect to a $1$-dimensional decomposition
$(S_n)_{n\geq 0}$. Let $\lambda:\Vr \to [0,1)$ a radial function which
satisfies one of the assertions of Theorem \ref{hitting-cont}.  
Then the unique radial function $W$ which satisfies $W(x_0)=1$ for all
$x_0\in S_0$ and $\tilde \Delta_m W(x)=\lambda(x)W(x)$ is given by 
\[
W(x)= \frac{\E_x\left[\exp \left( \int_0^{T_N} \lambda(X_s^x) ds
    \right)\right] }{\E_{\nu} \left[\exp \left( \int_0^{T_N}
      \lambda(X_{s}^\nu) ds \right)\right] },
 \textrm{ for  } |x|\leq N \textrm{ and } x_0 \in S_0,
\]
where $\nu$ is any probability measure supported on $S_0$ and the
hitting time $T_N:= 
\inf \{t \geq 0, X_{t}^\nu \in B_N^c \}$     of the set
$B_N^c$ for the  
   the continuous Markov process $(X_{t}^\nu )_{t\geq 0}$  on  $\Vr$
   whose generator is $-\Delta_m$ and initial law  is $\nu$. 
   
In particular, this function $W$ is a non-increasing radial positive function.
\end{proposition}

\begin{proof}
Actually, the only thing to prove is that the function $W$ in the proposition is well-defined.
For $N\geq 1$, consider the functions $W_N(x):=\E_x\left[\exp \left(
    \int_0^{T_N} \lambda(X_s^x) ds \right)\right]$. By hypothesis,
these functions are well-defined.  
Since $\Gr$ is weakly spherically symmetric, Proposition
\ref{prop-m-1d} ensures that $(|X_t|)_{t\geq0}$ is a continuous time
Markov process. Therefore  $W_N$ is a radial function. Moreover, it is
constant on $S_0$, thus we have  
$W_N(0)
=\E_{\nu} \left[\exp \left( \int_0^{T_N} \lambda(X_s^{\nu}) ds \right)\right]$ for  any probability measure $\nu$ supported on $S_0$.

Write also
\[\tilde W_N:=\frac{W_N}{W_N(0)} \mbox{ and } \tilde W_{N+1}: =
\frac{W_{N+1}}{W_{N+1}(0)}\]
so that $\tilde W_N(0)= \tilde W_{N+1}(0)=1$. 
Previous computations show that 
\[
\tilde \Delta_m \tilde  W_{N+1}(x) =  \lambda (x)  \tilde W_{N+1} (x)
\textrm{ for all } x\in B_N\]
and 
\[\tilde \Delta_m \tilde  W_{N}(x) =  \lambda (x)  \tilde W_{N} (x) \textrm{
  for all } x\in B_{N} 
\]
By lemma \ref{1dim}, we have $\tilde W_{N+1}(x)= \tilde W_N(x)$ for all $x\in B_{N}$.

It follows that the function $W$ in the proposition is well-defined
and satisfies $W>0$ and $\tilde \Delta W(x)=\lambda (x) W(x)$. It is clearly
radial as a limit of radial functions 
 and non-increasing by Lemma \ref{1dim}.
\end{proof} 

\begin{remark}
 In the case of the normalized Laplacian, the function $W$ in Proposition \ref{prop-lafctW-S0} can also be written as:
\[
W(x)= \frac{\E_x[A_{ T_N}] }{\E_{\nu} [A_{ T_N}] } \textrm{ for }|x|\leq N
\]
with $A_n:= \prod_{k=0}^{n-1} \Lambda(X_k^x)$,
$\Lambda(x):=\frac{1}{1-\lambda(x)}$, $\nu$ any probability measure on
$S_0$, and $T_N:= \inf \{n \geq 0, X_n^x \in B_N^c \}$ the hitting time
of the set $B_N^c$  
for $(X_k^\nu)_{k\geq 0}$  the random walk on $\Vr$ whose generator is $-\Delta_\eta$ and initial law $\nu$.
\end{remark}

\section{The bottom of the spectrum 
and of the essential spectrum}\label{s:10}   
In this section, we compare the bottom of the spectrum and the
essential spectrum  of  different weighted Laplacians.  
The idea here is to compare directly  the associated stochastic Markov
processes (see Proposition \ref{good-coupling}).  We then obtain  a
general comparison result (see Theorem \ref{thm-comparaison-general}).  
 This result  is an important improvement of Theorem 4  in \cite{KLW}

First, we provide a  coupling  between the Markov processes
on two different weighted graphs. 

\begin{proposition}\label{good-coupling}
Let $\Gr:=(\Vr^\Gr,\Er^\Gr,m^\Gr)$ and $\Hr:=(\Vr^\Hr,\Er^\Hr,m^\Hr)$
be two weighted graphs. Let $(S_n^\Gr)_{n\geq0}$ and
$(S_n^\Hr)_{n\geq0}$ be 1-dimensional decompositions for respectively
$\Gr$ and $\Hr$. Let $x_0\in \Gr$ and $y_0\in \Hr$ be such that
$|x_0|^\Gr=|y_0|^\Hr$. Let $(X_t^{x_0})_{t\geq 0}$ and
$(Y_t^{y_0})_{t\geq 0}$ be the continuous time Markov chains associated to
$\Delta_\Gr$ and $\Delta_\Hr$, respectively. 

Assume that for all $n\geq 0, x\in S_n^\Gr, y\in S_n^\Hr$ there exist $\deg_{0,0}^\Gr(x)\geq 0$ and $\deg_{0,0}^\Hr(y)\geq 0$
such that 
\begin{equation}\label{tilde-p}
 \tilde p_+^\Gr(x)\geq\tilde p_+^\Hr(y), \tilde p_-^\Gr(x)\leq \tilde p_-^\Hr(x) \textrm{ and } \widetilde \deg^\Gr(x)\geq  \widetilde \deg^\Hr(y);
\end{equation}
where
\[
\widetilde\deg^\Ar(z):=\deg^\Ar(z)+\deg_{0,0}^\Ar(z) \textrm{ and } \tilde p_l^\Ar(z):=\frac{\deg_l^\Ar(z)}{ \widetilde \deg^\Ar(z)},
\]
for $ z=x,y ; l=+,- \textrm{ and } \Ar=\Gr,\Hr$.
Then there exists a coupling of the processes $(X_t)_{t\geq 0}$ and $(Y_t)_{t\geq 0}$ such that, almost surely, 
\[
 |X_{G_i}|^\Gr = |Y_{H_i}|^\Hr, 
\]
\[
 |X_t|^\Gr \geq |Y_s|^\Hr  \textrm{ for }  t\in [G_i,G_{i+1}[,
 s\in[H_i,H_{i+1}[, \, i\geq 0,    
\]
where $[G_i,G_{i+1}[$ and $[H_i,H_{i+1}[$ are random intervals such
that, almost surely,  $G_i \to e(X^x), H_i \to  e(Y^y)$ as $i\to
+\infty$ and 
\[
 H_{i+1}-H_{i} \geq G_{i+1}-G_{i} \geq 0, \, i\geq 0. 
\]
Since $G_0=H_0=0$, almost surely, we have 
\begin{equation}\label{e:explosion}
 e(Y^y) \geq e(X^x),
\end{equation}

\begin{equation}\label{e:T_N}
T_{N,x_0}^\Gr \leq T_{N,y_0}^\Hr
\end{equation}
 where $T_{N,z_0}^\Ar:=\inf \{t\geq 0, |Z_t^{z_0}|^\Ar>N\}$, $\Ar=\Gr,\Hr; Z_t^{z_0}=X_t^{x_0}, Y_t^{y_0}$; and 
\begin{equation}\label{e:L_n}
L_{X^{x_0}}^N(n) \leq L_{Y^{y_0}}^N(n), \, 1\leq n\leq N
\end{equation}
where $L_{Z_t^{z_0}}^N(n):=\int_0^{T_N^\Ar} \bone_{S_n^\Ar}(Z_s^{z_0})
ds$, $Z_t^{z_0}=X_t^{x_0}, Y_t^{y_0}; \Ar=\Gr,\Hr;$ is the time spent
in the sphere $S_n^\Ar$ by the process $Z_t^{z_0}$ before it reaches
$S_{N+1}^\Ar$. 
\end{proposition}

\begin{proof}
We proceed by induction on $i\geq 0$. 
Assume $X_{G_i}=x, Y_{H_i}=y$ with $|x|^\Gr=|y|^\Hr$. Let us  add the
artificial jumps $\deg_{0,0}^\Gr(x)$ and $\deg_{0,0}^\Hr(y)$.  
Consider $G$ an independent exponential random variable of parameter
$\widetilde \deg^\Gr(x)$ and set  
$H:=\frac{\widetilde \deg^\Gr(x)}{\widetilde \deg^\Hr(y)} G$. $H$ is
thus an exponential random variable of parameter $\widetilde
\deg^\Hr(y)$. Clearly by construction $H\geq G$. 
Moreover, we can couple $X_t$ and $Y_s$ in such a way that after the jumps 
\[
|X_{G_i+G}|^\Gr \geq |Y_{H_i+H}|^\Hr.
\]
The construction will be explained below.
We then set $G_{i+1}:=G_i+G$.
If $|Y_{H_i+H}|^\Hr =|X_{G_{i+1}}|^\Gr$ we set $H_{i+1}=H_i+H$.
Otherwise if $|Y_{H_i+H}|^\Hr <|X_{G_{i+1}}|^\Gr$,  we freeze the
process $X$ in $X_{G_{i+1}}$ and let evolve independently the process
$Y$ until the time $s'$ defined by $s':=\inf \{u\geq H_i+H,
|Y_u|^\Hr=|X_t|^\Gr \}$.  
$s'$ is thus the hitting time of the sphere $S_{|X_t|^\Gr}^\Hr$. Since
$B_N$ is a finite set $s'$ is finite almost surely. We then set
$G_{i+1}=s'$.  

 Now we turn to the  construction of  the coupling of the jumps. Label
 the neighbors of $x$  and $y$  
by $x_1, \dots, x_r$  and $y_1, \dots, y_{r'}$ in such a way that
$|x_k|^\Gr$ and $|y_k|^\Hr$ are non-increasing with $k$. Note that  if
$\deg_{0,0}^\Ar(z)>0$ then $z$ is  a neighbor of $z$,
$\Ar=\Gr,\Hr;z=x,y$. 
Let  $U$ be an independent random variable with uniform law on $[0,1]$. 
 Set $X_{G_i+G}= x_j$ and $Y_{H_i+H}= y_j'$ where $j$ and $j'$ are
 the unique integer  in $\{1, \dots,r\}$ and  $\{1, \dots,r'\}$ such that  
\[
\tilde p^\Gr(x,x_1) +\dots
+ \tilde p^\Gr(x,x_{j-1}) \leq U < \tilde p^\Gr(x,x_1) +\dots + \tilde
p^\Gr(x,x_{j}) 
\]
and 
\[
\tilde p^\Hr(y,y_1) +\dots
+ \tilde p^\Hr(y,y_{j'-1}) \leq U < \tilde p^\Hr(y,y_1) +\dots +
\tilde p^\Hr(y,y_{j'}). 
\] 
Since by hypothesis $ \tilde p_+^\Gr(x)\geq\tilde p_+^\Hr(y)$ and $
\tilde p_-^\Gr(x)\leq \tilde p_-^\Hr(x)$, it is clear that  
\[
|X_{G_{i+1}}|^\Gr \geq |Y_{H_i+H}|^\Hr. 
\]
By using Lemma \ref{lemma-exp3}, $(X_t)_{t\geq0}$ and $(Y_s)_{s\geq0}$
are the Markov processes associated to $\Delta_\Gr$ and $\Delta_\Hr$. 
The other statements  are then immediate.
\end{proof}

Actually, there is a simpler characterization of condition
\eqref{tilde-p}.

\begin{definition}
Let $\Gr:=(\Vr^\Gr,\Er^\Gr,m^\Gr)$ and $\Hr:=(\Vr^\Hr,\Er^\Hr,m^\Hr)$ be two weighted graphs.  We say that $\Gr$ has a \emph{stronger
weak-curvature growth} than $\Hr$ if    
\begin{equation}\label{eq-condition-good-coupling}
\deg_+^\Gr(x)\geq \deg_+^\Hr(y) \textrm{ and } \frac{\deg_+^\Gr(x)}{\deg_-^\Gr(x)} \geq  \frac{\deg_+^\Hr(y)}{\deg_-^\Hr(y)}
\end{equation}
 for $x\in \Vr^\Gr, y\in \Vr^\Hr$, $|x|^\Gr=|y|^\Hr$.
\end{definition} 

\begin{proposition}\label{condition-good-coupling}
Let $\Gr:=(\Vr^\Gr,\Er^\Gr,m^\Gr)$ and $\Hr:=(\Vr^\Hr,\Er^\Hr,m^\Hr)$
be two weighted graphs.  
Then  \eqref{tilde-p} holds true if and only if  $\Gr$ has a 
stronger weak-curvature growth than $\Hr$.
\end{proposition}

\begin{proof}
Indeed,  condition \eqref{tilde-p} in Proposition \ref{good-coupling}
is equivalent to: there exist $z_1\geq \deg^\Gr(x)$ and $z_2\geq
\deg^\Hr(y)$ such that  
\[
\frac{\deg_+^\Gr(x)}{z_1}\geq \frac{\deg_+^\Hr(y)}{z_2},\, \frac{\deg_-^\Gr(x)}{z_1}\leq \frac{\deg_-^\Hr(y)}{z_2}
\textrm{ and } z_1\geq z_2.
\]
The above line is equivalent to 
\[
\max\left(\frac{\deg_-^\Gr(x)}{\deg_-^\Hr(y)},1\right) \leq \frac{z_1}{z_2} \leq \frac{\deg_+^\Gr(x)}{\deg_+^\Hr(y)}
\]
Therefore  condition \eqref{tilde-p} implies condition \eqref{eq-condition-good-coupling}.
Reciprocally, if  condition \eqref{eq-condition-good-coupling} holds, then it is possible to find $\alpha \geq 1$ such that 
\[
\max\left(\frac{\deg_-^\Gr(x)}{\deg_-^\Hr(y)},1\right) \leq \alpha \leq \frac{\deg_+^\Gr(x)}{\deg_+^\Hr(y)}
\]
It is then easy to see that one can choose $z_1\geq \deg^\Gr(x)$ and $z_2\geq \deg^\Hr(y)$ in  such a way  that $\frac{z_1}{z_2}=\alpha$.
\end{proof}

We then  state a general comparison result for the bottom of the spectra.
\begin{theorem}\label{thm-comparaison-general}
Let $\Gr:=(\Vr^\Gr,\Er^\Gr,m^\Gr)$ and $\Hr:=(\Vr^\Hr,\Er^\Hr,m^\Hr)$
two weighted graphs. Assume $\Gr$ has stronger weak-curvature growth
than  $\Hr$, then  
\[
 \inf \sigma (\Delta_\Gr) \geq \inf \sigma (\Delta_{\Hr}).
\]
If moreover $\inf\{ m^\Hr(x),x\in\Vr^\Hr\}>0$  or  $\Hr$ is a weakly symmetric graph and $m^\Hr(\Vr^\Hr)=+\infty$, then
\[
 \inf \sigma_{\rm ess} (\Delta_\Gr) \geq \inf \sigma_{\rm ess} (\Delta_{\Hr}).
\]
\end{theorem}

\begin{proof}
 We
keep the notation  of the proof of Proposition \ref{good-coupling}. 
Let $(X_t^x)_{t\geq 0}$ and $(Y_t^y)_{t\geq 0} $ be the coupled continuous
time Markov chains of generator  
  $\Delta_{\Gr}$ and $\Delta_{\Hr}$ starting in $x\in \Vr^\Gr$ and
  $y\in \Vr^\Hr$ such that $|x|^\Gr=|y|^\Hr$, respectively.

Let $\lambda_{\Hr,\rm ess}^0:=\inf \sigma_{\rm ess} (\Delta_\Hr)$ and
let $\ve>0$. With the hypothesis in Theorem
\ref{thm-comparaison-general}, by Theorem \ref{thm-reverse}, there
exist $n_0:=n_0(\ve)$ and  a positive function $W$ on $\Vr^\Hr$ such
that,  
\[
 \tilde \Delta_\Hr W(x)\geq \lambda(|x|^\Hr)  W(x).
\]
where
 $\lambda(n):= (\lambda_{\Hr,\rm ess}^0-\ve) \bone_{n\geq n_0}, n\geq 0$.
The probabilistic representation of Theorem \ref{hitting-cont} of
super-harmonic functions gives that for all $N\geq 1$, 
\[
 \E \left[\exp \left( \int_0^{T_N^{\Hr}}  \lambda(|Y_s^y|) ds
   \right)\right]<+\infty. 
\]
By Proposition \ref{good-coupling}, for $N\geq 1$ we have:
\[
L_{X^x}^N(n) \leq L_{Y^y}^N(n), \quad 0\leq n\leq N.
\]
By noticing that
\[
\E \left[\exp \left( \int_0^{T_N^{Z^z}}  \lambda(|Z_s^z|) ds
  \right)\right]= \E \left[\exp \left(\sum_{k=0}^{N} \lambda(n)  L_{Z^z}^N(n) 
  \right) \right], 
\]
with $Z^z=X^x$ or $Y^y$, this yields:
\[
 \E \left[\exp \left( \int_0^{T_N^\Gr}  \lambda(|X_s^x|) ds \right)\right]\leq
 \E \left[\exp \left( \int_0^{T_N^\Hr}  \lambda(|Y_s^y|) ds \right)\right]<+\infty.
\] 
Then the probabilistic representation of Theorem \ref{hitting-cont} gives the existence of a positive super-harmonic $\tilde W$ on $\Vr^\Gr$ satisfying 
$\tilde \Delta_{\Gr,m} \tilde W(x) \geq \lambda (|x|)\tilde W(x), x\in \Vr$.
Corollary \ref{cor-lyap-spec} and letting $\ve \to 0$ finally imply
\[
 \inf \sigma_{\rm ess} (\Delta_\Gr) \geq \inf \sigma_{\rm ess} (\Delta_\Hr).
\]
The proof  for the bottom of the spectrum is similar.
\end{proof}

\begin{remark} When $\Hr$ is a weakly spherically symmetric graph, we give a 
direct proof of Theorem \ref{thm-comparaison-general}. 
Let $\varepsilon>0$. By Theorem \ref{thm-reverse}, there exist
$n_0:=n_0(\ve)$ and  a positive non-increasing function $W$ on $\N$
such that,  
\[
\tilde \Delta_\Hr W(|y|^\Hr)\geq \lambda(|y|^\Hr)  W(|y|^\Hr).
\]
where
 $\lambda(n):= (\lambda_{\Hr,\rm ess}^0-\ve) \bone_{n\geq n_0}, n\geq 0$.
For $\Ar=\Gr,\Hr$, let $\Delta_{\bd,\Ar}$ be defined by 
\[
 \Delta_{\bd,\Ar}:= \frac{1}{\widetilde \deg^\Ar(\cdot)} \Delta_{\Ar}.
\]
For $f$ a radial function on $\Hr$, we have
\[
\tilde \Delta_{\bd,\Hr}f(n):=\tilde p_+^\Hr(n) (f(n)-f(n+1)) + \tilde
p_-^\Hr(n) (f(n)-f(n-1)). 
\]
Therefore, for $x\in \Vr^\Gr$, $y\in \Vr^\Hr$ such that $|x|^\Gr=|y|^\Hr=n$, 
\begin{align*}
 \tilde\Delta_{\Hr}W(|y|^\Hr)= \widetilde \deg^\Hr(y) \tilde
 \Delta_{\bd,\Hr}W(|y|^\Hr) \leq   \widetilde \deg^\Gr(x) \tilde
 \Delta_{\bd,\Gr}W(|x|^\Gr) = \tilde \Delta_{\Gr}W(|x|^\Gr). 
\end{align*}
Indeed we have $\widetilde \deg^\Hr(y)\leq \widetilde \deg^\Gr(x)$ and $\tilde
\Delta_{\bd,\Hr}W(|y|^\Hr)\leq \tilde \Delta_{\bd,\Gr}W(|x|^\Gr)$ since
$\tilde p_+^\Gr(x)\geq \tilde p_+^\Hr(y)$, $\tilde p_-^\Gr(x)\leq
\tilde p_-^\Hr(y)$ and $W$ is non-increasing. 
The end of the proof is now the same as before.
\end{remark}

\begin{remark}
Theorem \ref{thm-comparaison-general} is an improvement of
\cite[Theorem 4]{KLW} in three important directions.  First,
contrary to the latter, our result applies also to the bottom
of the essential spectrum. Second, they suppose that one of the graph
is weakly spherically symmetric. Third, our hypothesis is weaker, even 
for the comparison of the bottom of spectra. Therefore, the condition
called \emph{stronger curvature growth} introduced in \cite{KLW} and
which can be written as  
\beq\label{curvature-growth} 
\deg_+^\Gr(x)\geq \deg_+^\Hr(y) \textrm{ and } \deg_-^\Gr(x) \leq
\deg_-^\Hr(y) , \, x\in \Vr^\Gr, y\in \Vr^\Hr, \,  |x|^\Gr=|y|^\Hr 
\end{equation}
 is not the optimal one to compare the bottom of the spectrum.
\end{remark}

Here, we first recall that 
\[
 \inf \sigma(\Delta_{\Tc_d,1})=\inf \sigma_{\rm ess}( \Delta_{\Tc_d,1}) = d+1-2\sqrt{d}
\]
and 
\[
 \inf \sigma(\Delta_{\Tc_d,\eta})=\inf \sigma_{\rm ess}( \Delta_{\Tc_d,\eta}) =  1-\frac{2\sqrt{d}}{d+1}.
\]
where $\Tc_d$ denotes the simple $d$-ary tree.

\begin{example}
Let $\Gr:=(\Vr,\Er)$ be a simple tree such that  each vertex $x$ satisfies
$\eta_+(x)\in\{\alpha,\beta\}$ with $\alpha\leq \beta$, see Figure
\ref{f:ex107}.
A direct application of Theorem \ref{thm-comparaison-general} gives
\[
\inf \sigma( \Delta_{\Tc_\alpha,1})\leq \inf \sigma( \Delta_{\Gr,1}) \leq
\inf \sigma_{\rm ess}( \Delta_{\Gr,1})  \leq  \inf \sigma_{\rm ess}(
\Delta_{\Tc_\beta,1})=  \inf \sigma(
\Delta_{\Tc_\beta,1})
\] 
and 
\[
\inf \sigma( \Delta_{\Tc_\alpha,\eta})\leq \inf \sigma( \Delta_{\Gr,\eta})
\leq  \inf \sigma_{\rm ess}( \Delta_{\Gr,\eta})  \leq  \inf
\sigma_{\rm ess}( \Delta_{\Tc_\beta,\eta})=  \inf\sigma( \Delta_{\Tc_\beta,\eta}),
\]
where $\Tc_\alpha$ (resp.\ $\Tc_\beta$) denotes
 the simple $\alpha$-ary
(resp.\ $\beta$-ary) tree. Only the part with
$\Delta_{\Gr,1}$ was covered by \cite[Theorem 4]{KLW}.
\end{example}

\renewcommand\a{1}
\renewcommand\b{1}
\renewcommand\c{1}
\begin{figure}
\begin{tikzpicture}
\fill[color=black](0,0)circle(.7mm);
\foreach \x in {{-\a},0,\a} {
\draw (0,0)--(3*\x,1);
\fill[color=black](3*\x,1)circle(.7mm);
\foreach \y in {{-\b},\b} {
\fill[color=black](3*\x+\y/2,2)circle(.5mm);
\draw (3*\x,1)--(3*\x+\y/2,2);
\foreach \z in {{-\c},\c} {
\fill[color=black](3*\x+\y/2+\z/6,3)circle(.3mm);
\draw (3*\x+\y/2,2)--(3*\x+\y/2+\z/6,3);
}
}
}
\fill[color=black](-3.5,3)circle(.3mm);
\draw(-3.5,3)--(-3.5, 2);
\fill[color=black](0.5,3)circle(.3mm);
\draw(.5,3)--(.5, 2);
\fill[color=black](3.5,3)circle(.3mm);
\draw(3.5,3)--(3.5, 2);
\fill[color=black](2.5,3)circle(.3mm);
\draw(2.5,3)--(2.5, 2);
\fill[color=black](0,2)circle(.5mm);
\draw(0,2)--(0, 1);
\fill[color=black](1/6.25,3)circle(.3mm);
\draw(0,2)--(1/6.25,3);
\fill[color=black](-1/6.25,3)circle(.3mm);
\draw(0,2)--(-1/6.25,3);
\path(-5, 0) node {$S_0$};
\path(-5, 1) node {$S_1$};
\path(-5, 2) node {$S_2$};
\path(-5, 3) node {$S_3$};
\end{tikzpicture}
\caption{{Tree with $2$ or $3$ sons at each generation}}\label{f:ex107}
\end{figure}
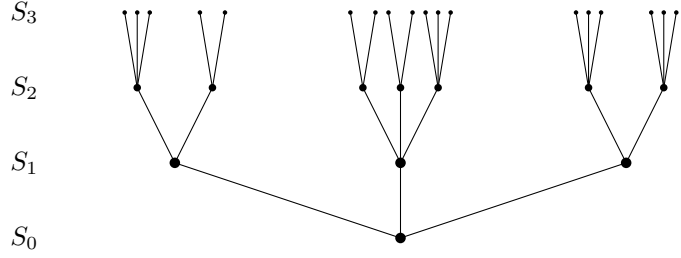

  Note that in the setting of the previous example, it is possible to have:
\[
\inf \sigma( \Delta_{\Gr,\eta}) < \inf \sigma_{\rm ess}( \Delta_{\Gr,\eta})
\] and an infinite number of eigenvalues below the  essential spectrum
can occur
(see \cite{Surchat}).

As a corollary of Theorem \ref{thm-comparaison-general}, we extend
Corollary 6.7 in \cite{KLW} to the case of the essential spectrum. 
\begin{corollary}\label{cor-spectrum-wss}
Let $\Gr$ and $\Hr$ be two weakly spherically symmetric graphs which
have the \emph{same curvature growth} in the sense of \cite{KLW}; that is: 
\[
\deg_+^\Gr(x)= \deg_+^\Hr(y) \textrm{ and } \deg_-^\Gr(x)=  \deg_-^\Hr(y) , \, x\in \Vr^\Gr, y\in \Vr^\Hr, \,  |x|^\Gr=|y|^\Hr. 
\]
Then
\[
 \inf \sigma (\Delta_\Gr) = \inf \sigma (\Delta_{\Hr}) \textrm{ and }  \inf \sigma_{\rm ess} (\Delta_\Gr) = \inf \sigma_{\rm ess} (\Delta_{\Hr}).
\]
\end{corollary}

In particular, on a fixed simple weakly spherically symmetric graph,
the bottoms of the spectrum and the essential spectrum of the
combinatorial Laplacian $\Delta_1$ 
do not change if one adds or removes edges inside the spheres $S_n$.

\begin{example}
Let $\Gr:=(\Vr,\Er)$ be a simple infinite bipartite graph and  $x_0
\in \Vr$ such that $\eta_+(x_0)=2, \eta_-(x_0)=0$ and  
\begin{align}\label{e:line}
(\eta_+(x)=2, \eta_-(x) =1) \textrm{ or } (\eta_+(x)=4, \eta_-(x)=2) ,
\, x\in \Vr, x\neq x_0. 
\end{align}
Then $(\Vr,\Er, \eta)$ is a weakly spherically symmetric graph. Thus 
Corollary \ref{cor-spectrum-wss} gives that:
\[
\inf \sigma( \Delta_{\Tc_2,\eta})= \inf \sigma( \Delta_{\Gr,\eta}) =
\inf \sigma_{\rm ess}( \Delta_{\Gr,\eta})  =  \inf \sigma_{\rm ess}(
\Delta_{\Tc_2,\eta}). 
\]
and Theorem \ref{thm-comparaison-general} that:
\[
\inf \sigma( \Delta_{\Tc_2,1})\leq \inf \sigma( \Delta_{\Gr,1}) \leq
\inf \sigma_{\rm ess}( \Delta_{\Gr,1})  \leq   2 \inf \sigma_{\rm
  ess}( \Delta_{\Tc_2,1}). 
\]
Note that, if moreover one supposes that
\[
\min \{ \deg_+ (x), x\in S_n\} \leq \max \{\deg_-(x), x\in S_n\},
\]
a direct application of \cite[Theorem 4]{KLW} gives only 
that $\inf \sigma(\Delta_{\Gr,1})\geq 0$.
\end{example}

\renewcommand\a{1}
\renewcommand\b{1}
\renewcommand\c{1}
\begin{figure}
\begin{tikzpicture}
\fill[color=black](0,0)circle(.7mm);
\draw (0,0)--(-1,1);
\draw (0,0)--(1,1);
\fill[color=black](1,1)circle(.7mm);
\fill[color=black](-1,1)circle(.7mm);
\draw (0,2)--(-1,1);
\draw (-2,2)--(-1,1);
\draw (0,2)--(1,1);
\draw (2,2)--(1,1);
\fill[color=black](-2,2)circle(.7mm);
\fill[color=black](0,2)circle(.7mm);
\fill[color=black](2,2)circle(.7mm);
\draw (0,2)--(-1,3);
\draw (0,2)--(-.33,3);
\draw (0,2)--(1,3);
\draw (0,2)--(.33,3);
\fill[color=black](-1,3)circle(.7mm);
\fill[color=black](-.33,3)circle(.7mm);
\fill[color=black](.33,3)circle(.7mm);
\fill[color=black](1,3)circle(.7mm);
\draw (-2,2)--(-1,3);
\draw (-2,2)--(-3,3);
\fill[color=black](-3,3)circle(.7mm);
\draw (2,2)--(-3,3);
\draw (2,2)--(2,3);
\fill[color=black](2,3)circle(.7mm);
\draw (-3.33,4)--(-3,3);
\draw (-2.67,4)--(-3,3);
\draw (-3.33+.22,4)--(-3,3);
\draw (-2.67-.22,4)--(-3,3);
\fill[color=black](-3.33,4)circle(.5mm);
\fill[color=black](-3.33+.22,4)circle(.5mm);
\fill[color=black](-2.67,4)circle(.5mm);
\fill[color=black](-2.67-.22,4)circle(.5mm);
\draw (-1.33,4)--(-1,3);
\draw (-0.67,4)--(-1,3);
\draw (-1.33+.22,4)--(-1,3);
\draw (-0.67-.22,4)--(-1,3);
\fill[color=black](-1.33,4)circle(.5mm);
\fill[color=black](-1.33+.22,4)circle(.5mm);
\fill[color=black](-0.67,4)circle(.5mm);
\fill[color=black](-0.67-.22,4)circle(.5mm);
\draw (+1.33-.22,4)--(1,3);
\draw (0.67+.22,4)--(1,3);
\fill[color=black](1.33-.22,4)circle(.5mm);
\fill[color=black](0.67+.22,4)circle(.5mm);
\draw (+2.33-.22,4)--(2,3);
\draw (1.67+.22,4)--(2,3);
\fill[color=black](2.33-.22,4)circle(.5mm);
\fill[color=black](1.67+.22,4)circle(.5mm);
\draw (+1.33-.44,4)--(.33,3);
\draw (-2.67,4)--(.33,3);
\draw (+1.33-.22,4)--(-.33,3);
\draw (-3.33,4)--(-.33,3);
\path(-5, 0) node {$S_0$};
\path(-5, 1) node {$S_1$};
\path(-5, 2) node {$S_2$};
\path(-5, 3) node {$S_3$};
\path(-5, 4) node {$S_4$};
\end{tikzpicture}
\caption{{A weakly spherically symmetric graph satisfying
    \eqref{e:line}}}
\end{figure}
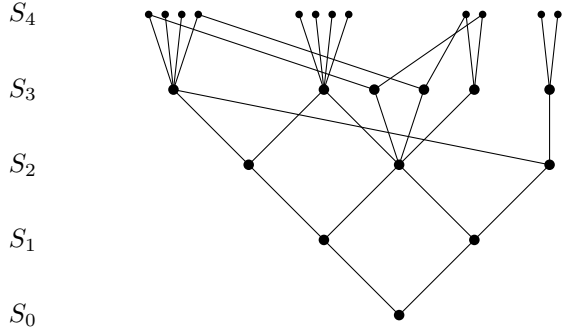

\section{Stochastic completeness}\label{s:11}
By definition, a graph $\Gr:=(\Vr^\Gr,\Er^\Gr,m^\Gr)$ is said to be
\emph{stochastically complete}  
if for all $x\in \Vr$, $\PR( e(X^x)<+\infty)=0$ where $X^x$ is the
minimal right continuous Markov process constructed in section
\ref{s:markov-process}. Otherwise, it is said to be
\emph{stochastically incomplete}. 
First we consider weakly spherically symmetric graph. The following
result was already included in  Theorem 5 in \cite{KLW} (except that
we slightly generalize their notion of spherically symmetric graphs). 
We provide a direct stochastic proof. 
\begin{theorem}\label{thm-1-dim-SC}
Let $\Gr:=(\Vr^\Gr,\Er^\Gr,m^\Gr)$ be a \emph{weakly spherically
  symmetric graph} and let $\Gr_{\N}:=(\N, \Er_\N, 
m_\N)$ where $\Er_\N$ and $m_\N$ are defined as in Proposition \ref{prop-m-1d}. 
Then $\Gr$ is \emph{stochastically complete} if and only if $\Gr_\N$ is.
\end{theorem}

\begin{proof}
Let $n\geq0$ and $x \in S_n$ and let  $(X^x_t)_{t\geq 0}$ be the 
minimal right continuous Markov process associated to $\Gr$. 
By Proposition \ref{prop-m-1d}, $(|X^x_t|)_{t\geq 0}$ is the minimal
right continuous Markov process associated to $\Gr_\N$ and starting in
$n$. 
Clearly, one has $e(X^x)=e(|X^x|)$. The conclusion of the theorem
follows by the definition of stochastic completeness. 
 \end{proof}

The coupling argument of Proposition \ref{good-coupling} implies the following comparison result. It is an improvement of Theorem 6 in\cite{KLW}.
\begin{theorem}\label{thm-comparaison-SC}
Let $\Gr:=(\Vr^\Gr,\Er^\Gr,m^\Gr)$ and $\Hr:=(\Vr^\Hr,\Er^\Hr,m^\Hr)$
two weighted graphs. Assume $\Gr$ has stronger weak-curvature growth
than $\Hr$.  
If $\Hr$ is stochastically incomplete then so is $\Gr$.
If  $\Gr$ is stochastically complete then so is $\Hr$.
\end{theorem}

\begin{proof}
 Let $n\geq 0$, $x \in S_n(\Gr)$ and $y\in S_n (\Hr)$. Proposition \ref{good-coupling} provides a coupling  $(X^x_t)_{t\geq 0}$ and $(Y^y_t)_{t\geq 0}$
 of the  two minimal right continuous Markov chains on $\Gr$ and $\Hr$ starting in $x$ and $y$, respectively, such that  $e(X^x) \leq e(Y^y)$. The conclusion of the theorem follows by the definition of stochastic completeness.
\end{proof}
\appendix
\section{The Friedrichs extension}\label{s:friedrichs} 
\setcounter{equation}{0} 

In this section, we recall the construction of the Friedrichs
extension of a positive symmetric densely defined operator.  
Given a dense subspace $\Dr$ of a Hilbert space $\Hr$ and a
non-negative symmetric operator $H$ on $\Dr$, let $\Hr_1$ be the
completion of $\Dr$ under the norm given by $\Qr(\varphi)^2=\langle
H\varphi, \varphi\rangle+\|\varphi\|^2$.  The domain of the
Friedrichs extension of $H$ is given by 
\begin{align*}
\Dc(H^\Fr)&=\{f\in \Hr_1 \mid
\Dr\ni g\mapsto \langle Hg,f\rangle+\langle g,f\rangle \mbox{ extends
  to a norm continuous} 
\\
&\quad \quad \mbox{ function on } \Hr\}
\\
&= \Hr_1\cap \Dc(H^*).
\end{align*}
For each
$f\in\Dc(H^\Fr)$, there is a unique $u_f$ such that $\langle
Hg,f\rangle + \langle g,f\rangle  = \langle g,u_f\rangle$, by Riesz'
Theorem. The Friedrichs extension of $H$, is given by $H^\Fr
f:=u_f-f$. It is a self-adjoint extension of $H$, e.g., \cite[Theorem
X.23]{RS}.  

We now describe the domain of the adjoint of the discrete
Laplacian. This is well-known, e.g., \cite{CTT, KL2}. Let  $\Gr=(\Vr,
\Er, m)$ be a 
weighted graph. We have:
\begin{align*}
  \Dc\bigl((\Delta_{\Gr}|_{\Cc_c(\Vr)})^*\bigr)=\Big\{f\in \ell^2(\Vr, m),&
\\
&\hspace*{-2cm}   x\mapsto \frac{1}{m(x)} \sum_{y\in \Vr} \Er(x,y)
   (f(x)- f(y))\in\ell^2(\Vr, m)\Big\}.
\end{align*}
Then, given $f\in\Dc((\Delta_{\Gr}|_{\Cc_c(\Vr)})^*)$, one has:
\begin{eqnarray*}
  \left((\Delta_{\Gr}|_{\Cc_c(\Vr)})^*
    f\right)(x)=\frac{1}{m(x)}\sum_{y\in \Vr}\Er(x,y)(f(x)-f(y)),
\end{eqnarray*}
for all $x\in \Vr$. We recall that $\Hr_1$ here is the completion of
$\Cc_c(\Vr)$ under the norm:
\[\|f\|_{\Hr_1}^2:= \frac{1}{2}\sum_{x,y \in \Vr}\Er(x,y)|f(x)-f(y)|^2+\|f\|^2.\]
By definition, the operator
$\Delta_{\Gr}|_{\Cc_c(\Vr)}$ is
\emph{essentially self-adjoint} if its closure is equal to its
adjoint. A review of recent developments of essential
  self-adjointness may be found in \cite{Go2}.

\section{Lyapunov functions, Super-Poincar\'e inequality and the
  infimum of the essential spectrum}\label{section-SPI} 
\setcounter{equation}{0}  

In this section, we explain how to use the above Lyapunov functions
so as to obtain a lower bound on the infimum of the essential spectrum and
in some cases its emptiness. This is a straightforwardly adaptation of
the continuous setting.   In our discrete setting, this
  approach is not strictly necessary  to obtain our results (see
  Remark \ref{r:A7}) on the essential spectra. We have included it for
  the sake of completeness and because it provides some insights and  a
  characterization  for the lower bound of the essential spectrum.
We shall rely on the following Super-   Poincar\'e Inequality, which
was introduced by Wang (see \cite{Wan1, Wan2, Wan3}).

\begin{definition}
 We say that \emph{Super-  Poincar\'e Inequality of parameter} $s_0\in
 \R$ holds true, if  there is a function
$h:\Vr\to(0,\infty)$ such that $m(|h|^2)=1$ and some positive
non-increasing functions $\beta_h:(s_0,\infty)\to(0,\infty)$, such that
\begin{equation}\label{SPI}
  \mbox{SPI }(s_0)\quad \quad m(|f|^2) \leq s \,  m\left( \overline
  f\,\Delta_m f\right) + 
  \beta_h(s) \, m(|f| h)^2 \textrm{ for all } s>s_0 
\end{equation}
and $f\in \Cc_c(\Vr)$. 
\end{definition} 
 
Where, by abuse of notation, we  wrote:
\[m(f):=\sum_{x\in \Vr} f(x) m(x).\]

\begin{remark}\label{beta-decr}
 If SPI $(s_0)$ holds for some not necessarily non-increasing function
 $\beta_h$, then, for any $\mu>s_0$, SPI $(\mu)$  holds for the
 non-increasing function  $\beta'_h(s)=\inf_{\mu\leq s}\beta_h(s)$. 
\end{remark}

We start by showing how the existence of Lyapunov functions implies SPI.

\begin{theorem}\label{thm-SPI}
Take $\psi$ a  non-negative function, $b\geq 0$, and some finite set
$B_{r_0}$. Denote by 
\[
\Psi(r)= \inf \{\psi (x), x\in B_r^c\} \textrm{ and } s_0 =
\left(\lim_{r\to \infty} \Psi(r)\right)^{-1} \in [0, \infty).  
\]
Suppose that $\Psi(r)>0$ for $r$ large enough and that $W$ is a positive
function such that 
\begin{equation}\label{lyap-W}
\Delta_m W \geq \psi \times W - b \bone_{B_{r_0}} 
\end{equation}
Let $h:\Vr\to(0,\infty)$ such that $\|h\|_m=1$ and   let $a_h(r)=\inf
\{h(x)^2 m(x), x\in B_r\} $.  
Then SPI $(s_0)$ holds true with 
\[\beta_h (s)=
\left(1+\frac{bs}{\alpha_{r_0} }\right) \frac{1}{a_h(
  \Psi^{-1}(\frac{1}{s}))},\] 
where $\alpha_{r_0}=\inf\{W(x),x\in B_{r_0}\}$. 
\end{theorem}
We follow the proof in \cite{CGWW}[Section 2]. 

\begin{proof}
Set $r\geq r_0$ such that $\Psi(r)>0$:
\begin{align*}
\Vert f\Vert^2 = m (|f|^2) =& m (\frac{|f|^2 \psi}{\psi}
\bone_{B_r^c}) + m (|f|^2  \bone_{B_r}) 
                           \leq  \frac{1}{\Psi(r)} m (|f|^2 \psi
                           \bone_{B_r^c}) + m (|f|^2  \bone_{B_r})\\ 
                            \leq & \frac{1}{\Psi(r)} m (|f|^2 \psi) + m
                            (|f|^2  \bone_{B_r})\\ 
                           \leq & \frac{1}{\Psi(r)} m \left(|f|^2
                             \left(\frac{\Delta_m W}{W} + \frac{b
                                 \bone_{B_{r_0}}}{W}\right)\right)  +
                           m (|f|^2  \bone_{B_r})\\ 
                          \leq & \frac{1}{\Psi(r)}m \left(|f|^2
                            \left(\frac{\Delta_m W}{W} \right)\right) +
                          \left(\frac{b}{\Psi(r) \inf \{W(x),x\in
                              B_{r_0}\}}+ 1\right) m (|f|^2
                          \bone_{B_r}).        
\end{align*}
We concentrate on the second term. Let $h:\Vr \to (0, \infty)$ such that
$m (h^2)=1$. Since the set $B_r$ is finite, 
\[
m(|f| h \bone _{B_r})^2 = \left(\sum_{x\in B_r} |f(x)| h(x) m(x)
\right)^2 \geq  \sum_{x\in B_r} |f(x)|^2 h(x)^2 m(x)^2 \geq \left( \inf
  _{x\in B_r} h(x)^2 m(x) \right) m (|f|^2 \bone_{B_r}) 
\] 
which gives the following local
Super-Poincar\'e Inequality: 
\[
m(|f|^2 \bone_{B_r}) \leq \left( \inf _{x\in B_r} m(x) h(x)^2 \right)^{-1} m(|f| h)^2. 
\]
Therefore, by combining the above estimate and the Hardy inequality
\eqref{Hardy}, we get:
\[
m(|f|^2) \leq \frac{1}{\Psi(r)} m(\, \overline f \Delta_m f) +
\left(\frac{b}{\Psi(r) \inf \{W(x),x\in B_{r_0}\} }+ 1\right) \left(
  \inf _{x\in B_r} 
  m(x) h(x)^2 \right)^{-1} m(|f| h)^2. 
\] 
Finally, this yields the SPI $(s_0)$ with $\beta_h(s)$ defined as in
the theorem. 
\end{proof}

\begin{remark}
Note that if the constant $b=0$ in the above, the function
$\beta_h(s)$ is then given by 
\[
\beta_h (s)= \frac{1}{a_h( \Psi^{-1}(\frac{1}{s}))}.
\]  
\end{remark}

We turn now to the  equivalence between the Super-Poincar\'e
Inequality and the infimum of the essential spectrum of the operator
(see Theorem 2.2 in \cite{Wan3}).  
\begin{theorem}\label{thm-SPI-bas-sp}
Let $s_0>0$.  Then the following assertions are equivalent:
\begin{enumerate}
 \item $\sigma_{\rm ess} (\Delta_m) \subset [\frac{1}{s_0}, \infty )$.
\item There exists a positive function $h$ such that $m(h^2)=1$ 
 and a non-increasing function $\beta_h:(s_0,\infty )\to(0,\infty )$ such that (\ref{SPI}) holds for all $s>s_0$.

\item For any positive function $h$ such that $m(h^2)=1$,
 there exists a non-increasing function $\beta_h:(s_0,\infty
 )\to(0,\infty )$ such that (\ref{SPI}) holds for all $s>s_0$. 
\end{enumerate}
In particular,  $\sigma_{\rm ess} (\Delta_m)=\emptyset$ if and only if
there exists some functions $h$ and $\beta_h$ for which (\ref{SPI})
holds for $s>0$. 
\end{theorem}
For the seek of completeness, we will give the proof of this result. We follow the proof of Theorem 3.2 in \cite{Wan2} which was originally made in the case of probability measure. The slight
difference comes from the fact that we consider the Friedrichs
extension and we are not in a essentially self-adjoint setting.
\begin{proof}
It is is clear that (c) implies (b). We show first that (b) implies
(a). Let $h$ be the positive function such that $m(h^2)=1$ and let
$f\in \Cc_c(\Vr)$. Let $0<\ve<1$, $\eta>0$ and let $B_r$ such that
$m(\bone_{B_r^c} h^2 )\leq \ve$, then, for $f\in \Cc_c(\Vr)$ such
that $m(|f|^2)=1$ and $f|_{B_r}=0$,   
\begin{eqnarray*}
1=m(|f|^2) &\leq & (s_0+\eta) \langle f,\Delta_m f \rangle_m +
\beta(s_0+\eta) m(|f| h )^2 \\ 
         &\leq &  (s_0+\eta) \langle f,\Delta_m f \rangle_m +
         \beta(s_0+\eta) m(|f|^2) m(\bone_{B_r^c} h^2 )\\ 
         &\leq &  (s_0+\eta) \langle f,\Delta_m f \rangle_m +
         \beta(s_0+\eta) \ve. 
\end{eqnarray*} 
Using \eqref{e:Persson}, we get
\[
\inf \sigma_{\rm ess} (\Delta_m) \geq \sup_{\ve>0,\eta >0} \frac{1-
  \beta(s_0+\eta) \ve}{s_0+\eta} =\frac{1}{s_0}. 
\]

Now we show that (a) implies (c). Let $h$ be a positive function
such that $m(h^2)=1$. Let $r'>r>s_0$. Since  $r>s_0$,
$\sigma (\Delta_m) \cap [0,\frac{1}{r}]$ 
 is given by a finite number of finite dimensional eigenvalues. 
Let $0\leq \lambda_1 \leq \dots \leq \lambda_{n_r}$ be these
eigenvalues (including multiplicity), $g_1, \dots , g_{n_r}$   be some
associated orthonormalized eigenfunctions, and $\mathcal H_r$ the
corresponding spanned vector space. 
Let $f\in\Cc_c(\Vr)$ and consider $g:= \bone_{[0,\frac{1}{r}]}(\Delta_m) f=
\sum_{i=1}^{n_r} m(\,\overline g_i f) g_i$  
and $k:= \bone_{(\frac{1}{r},\infty)}(\Delta_m) f$.
By construction $m(|f|^2)=m(|g|^2)+ m(|k|^2)$ and
\[m(|k|^2)\leq r\cdot m(\, \overline k \Delta_m k)\leq r\cdot m(\,
\overline f \Delta_m f).\]
Moreover, since $\mathcal H_r$ is finite dimensional, 
there is a finite $\beta_1(r)$ such that
\[
m(|u|^2) \leq \beta_1(r) m(|uh|)^2, \mbox{ for all } u\in \mathcal H_r.
\]
Let $c_r>0$ be a constant to be precised later. Then using several times
the Cauchy-Schwarz inequality, 
\begin{align*}
 m(|gh|) &\leq  \sum_{i=1}^{n_r} |m(f \overline g_i)| m(|h g_i|) \leq
 \sum_{i=1}^{n_r} m(|f g_i|)\\ 
          &\leq \sum_{i=1}^{n_r}\left( c_r m(|f|h)  +  m \left( | f g_i| \bone_{ \{|g_i| \geq c_r h\} }\right)\right)\leq n_r  c_r m(|f|h) + n_r  \ve_r^{1/2}  m( |f|^2)^{1/2},
\end{align*}
where $\ve_r:= \sup_{i=1,\dots,n_r} m \left(             |g_i|^2
  \bone_{\{|g_i|\geq c_r               h\}}\right)$.  
By dominated convergence theorem $\ve_r\to 0$ when $c_r \to
\infty$. Therefore 
\[
m(|g|^2) \leq 2\beta_1(r) \left(n_r^2 c_r^2 m(|fh|)^2 + n_r^2  \ve_r m(|f|^2)\right)
\]
and 
\[
\left( 1- 2\beta_1(r) n_r^2 \ve_r\right) m(|f|^2) \leq r\cdot m(\, \overline
f\Delta_m f) 
+   2\beta_1(r) n_r^2 c_r^2  m(|fh|)^2. 
\]
Taking  $r'>r$  and $c_r$ large enough such that $\ve_r \leq
\frac{r'-r}{2\beta_1(r) n_r^2 r'}$ gives 
\[
m(|f|^2) \leq r' m(\, \overline f\Delta_m f) +  2\beta_1(r) n_r^2 c_r^2
\frac{r'}{r} m(|fh|)^2. 
\]
Taking $\beta_h(r')= \inf_{s_0<r<r'} 2\beta_1(r) n_r^2 c_r^2
\frac{r'}{r}$ ends the proof. 
\end{proof}

The conjunction of Theorem \ref{thm-SPI} and Theorem
\ref{thm-SPI-bas-sp} gives the following result. 
\begin{corollary}\label{cor-lyap-spec}
Assume there exists $W$ a positive function such that  
\[
\tilde \Delta_m W \geq  \psi \times W - b \bone_{B_{r_0}} 
\]
for some non-negative function $\psi$, some constant $b\geq 0$ and
some finite set $B_{r_0}$.  

If  $\liminf \psi(x) =l$, as $|x|\to \infty$, then
$\sigma_{\rm ess} (\Delta_m)\subset[l,\infty)$. 
In particular, if $\lim \psi(x) =+\infty$, then  $\sigma_{\rm ess}
(\Delta_m)=\emptyset$. 
\end{corollary}

\begin{proof}
Indeed, with our assumptions, Theorem \ref{thm-SPI} gives 
SPI $(1/l)$ and Theorem \ref{thm-SPI-bas-sp} implies in turn that   
$\sigma_{\rm ess} ( \Delta_m)  \subset [l,\infty)$.  
\end{proof}

\begin{remark}\label{r:A7}
One can avoid the use of Super-Poincar\'e Inequality in our setting
and give a direct proof of Corollary \ref{cor-lyap-spec} by using the
Hardy inequality \ref{Hardy} and either the Persson Lemma or 
the min-max principle.
\end{remark} 

\noindent\textbf{Acknowledgement}.
MB was partially supported by the ANR project HAB
(ANR-12-BS01-0013-02). SG was partially supported by the ANR project
GeRaSic and SQFT.

\end{document}